\documentclass{article}
\usepackage{times}
\usepackage{epsfig}
\usepackage{graphicx}
\usepackage{amsmath}
\usepackage{amssymb}
\usepackage{amsthm}
\usepackage{tikz}
\usepackage{pgfplots}
\usepackage{bm}
\usepackage[T1]{fontenc}
\usepackage{geometry}
\usepackage{caption}
\usepackage{mathtools,verbatim}
\usepackage{wrapfig}
\usepackage{authblk}
\usepackage{multirow}
\usetikzlibrary{plotmarks}

\usepackage[pagebackref=true,breaklinks=true,letterpaper=true,colorlinks,bookmarks=false]{hyperref}

\def\tx{{\tilde{x}}}
\def\ty{{\tilde{y}}}
\def\tz{{\tilde{z}}}
\def\tI{{\tilde{I}}}
\def\btz{{\bm{ \tilde{z}}}}
\def\btx{{\bm{ \tilde{x}}}}
\def\bty{{\bm{ \tilde{y}}}}
\def\x{{\bm x}}
\def\eps{{\bm \epsilon}}

\def\y{{\bm y}}
\def\z{{\bm z}}
\def\w{{\bm w}}
\def\b{{\bm b}}

\def\s{{\bm s}}
\def\o{{\bm o}}
\def\a{{\bm a}}
\def\O{O}
\def\0{{\bm 0}}
\def\bto{{\bm{\tilde{o}}}}
\def\to{\tilde{o}}
\def\A{{\mathcal{A}}}

\def\rank{{\text{rank}}}
\def\supp{{\text{supp}}}
\def\card{{\text{card}}}

\newcommand{\skal}[1]{\langle #1 \rangle}

\newcommand{\reg}{\ensuremath{\mathcal{R}}}

\newtheorem{theorem}{Theorem}[section]
\newtheorem{lemma}[theorem]{Lemma}

\newtheorem{corollary}[theorem]{Corollary}

\title{Relaxations for Non-Separable Cardinality/Rank Penalties}
\date{}
\author{Carl Olsson$^{1,2}$ \hspace{5mm} Daniele Gerosa$^{1}$ \hspace{5mm} Marcus Carlsson$^{1}$ \\
	\begin{minipage}[c]{0.4\textwidth}
		\centering
		${}^1$Centre for Mathematical Sciences,\\
		Lund University
	\end{minipage}
	\begin{minipage}[c]{0.4\textwidth}
	\centering
	${}^2$Department of Electrical Engineering,\\
	Chalmers University of Technology
\end{minipage}
}

\DeclareMathOperator*{\argmin}{arg\,min}

\begin{document}

\maketitle

\begin{abstract}
Rank and cardinality penalties are hard to handle in optimization frameworks due to non-convexity and discontinuity. Strong approximations have been a subject of intense study and numerous formulations have been proposed. Most of these can be described as separable, meaning that they apply a penalty to each element (or singular value) based on size, without considering the joint distribution.
In this paper we present a class of non-separable penalties and give a recipe for computing strong relaxations suitable for optimization. In our analysis of this formulation we first give conditions that ensure that the globally optimal solution of the relaxation is the same as that of the original (unrelaxed) objective. We then show how a stationary point can be guaranteed to be unique under the RIP assumption (despite non-convexity of the framework).
\end{abstract}

\section{Introduction}
Sparsity and low rank priors are common ways of regularizing ill-posed inverse problems.
In the computer vision community they have been employed in a wide variety of applications such as
outlier detection/removal, face recognition, rigid and non rigid structure from motion, photometric stereo and optical flow  \cite{wright-etal-pami-2009,olsson-etal-cvpr-2010,tomasi-kanade-ijcv-1992,bregler-etal-cvpr-2000,yan-pollefeys-pami-2008,garg-etal-cvpr-2013,basri-etal-ijcv-2007,garg-etal-ijcv-2013}. 
The prior is typically formulated either as a soft penalty, resulting in a trade-off between data fit and regularization, or as a hard constraint enforcing a particular cardinality/rank.
In this paper we formulate the general sparsity regularized problem as
\begin{equation}
G(\card(\x)) + \|A \x - \b\|^2,
\label{eq:vectprobl}
\end{equation}
where the function $G$ can be written as
\begin{equation}
G(k) = \sum_{i=0}^k g_i, 
\end{equation}
where $0= g_0 \leq g_1 \leq g_2 \leq ... \leq g_n \leq \infty$. Note that $g_i=\infty$ is allowed (if $i>0$). This formulation covers both hard and soft priors. If we for example chose $G(k) = \mu k$ we get an objective that penalizes but does not restrict the sparsity of the solution. In contrast, if we let $G(k) = 0$ when $k \leq r$ and $\infty$ when $k > r$ we get a hard cardinality constraint. Many other choices for $G$ are possible. 
In this paper we will also consider the corresponding matrix version of \eqref{eq:vectprobl}, formulated as
\begin{equation}
G(\rank(X)) + \|\A X - \b\|^2,
\label{eq:matrixprobl}
\end{equation}
where $\A$ is a linear operator.
The theory for the vector and sparsity formulations are with a few exceptions very similar, since the rank of a matrix is basically a sparsity prior on the singular values of the matrix. We will therefore state our main results in the vector setting but emphasize that they apply for the matrix setting as well.

In general $G$ is convex and non-decreasing on the non-negative integers. However as function of the unknown $\x$, $G(\card(\x))$ is highly non-convex as well as discontinuous and 
in general these problems are NP-hard \cite{natarajan1995sparse,gillis-glineur-siam-2011}.
Therefore relaxations have to be employed.
In recent years there have been a lot of work on convex as well as non-convex relaxations for both sparsity and rank regularized problems. The standard method is to replace $\card(\x)$ with the convex $\ell_1$ norm $\|\x\|_1$ \cite{tropp-2015,tropp-2006,candes-tao-2006,donoho-elad-2002}. 
Furthermore, if the RIP constraint
\begin{equation}
(1-\delta_r)\|\x\|^2 \leq \|A \x\|^2 \leq (1+\delta_r)\|\x\|^2,
\label{eq:vectorRIP}
\end{equation}
holds for all vectors $\x$ with $\card(\x) \leq r$, asymptotic performance guarantees can be derived \cite{candes-tao-2006}.
On the other hand the $\ell_1$ approach suffers from a shrinking bias
since it penalizes both small elements of $\x$, assumed to stem from measurement noise, and large elements, assumed to make up the true signal, equally. Hence the suppression of noise also requires an equal suppression of signal \cite{fan2001variable,mazumder2011sparsenet}.
This insight has lead to a large number of non-convex alternatives able to penalize small components proportionally harder than the large ones e.g. \cite{fan2001variable,mazumder2011sparsenet,bredies2015minimization,blumensath2009iterative,chartrand2007exact,pan2015relaxed,zou2008one,loh2013regularized,zhang2012general,zhang2010nearly,loh2017support}. 
With a few exceptions e.g. \cite{loh2013regularized,loh-wainwright-arxiv-2014} global optimality guarantees are generally not available for these formulation. 
In addition these typically employ separable formulations, that is, a non-convex penalty is applied to each element without regarding the joint element values. Such a formulation can for example not add hard thresholds on the number of non-zero elements of the vector.
It is important to note that under RIP the matrix $A$ typically has a nullspace containing dense vectors. Under such conditions, separable formulations that don't have shrinking bias, often have local minimizers, see Section~\ref{sec:motivation}. Hence non-separable $G$ able to strongly penalize high cardinality solutions is likely to provide better relaxations. On the other hand these are harder to analyse and less common in the literature.
The k-support norm studied in \cite{argyriou-etal-nips-2012,mcdonald-etal-jmlr-2016} is a non-separable surrogate for the rank function. It is however a convex norm and therefore suffers from a shrinking bias similar to the $\ell_1$ norm.

The theory of rank minimization is largely analogous that of sparsity.
In this context the rank function is typically replaced with the convex nuclear norm $\|X\|_* = \sum_i \sigma_i(X)$ \cite{recht-etal-siam-2010,candes-etal-acm-2011}.
In \cite{recht-etal-siam-2010} the notion of RIP was generalized to the matrix setting.
A number of generalizations that give performance guarantees for the matrix case have appeared \cite{oymak2011simplified,candes-etal-acm-2011,candes2009exact} and non-convex alternatives have also been considered \cite{oh-etal-pami-2016,oymak-etal-2015,mohan2010iterative,iglesias2020accurate}.
The analogue of the k-support norm was considered in \cite{eriksson-etal-cvpr-2015,grussler-etal-tac-2018}.  
The so called weighted nuclear norm is a popular choice for vision problems \cite{hu-etal-pami-2013,gu-2016,kumar-arxiv-2019}.
We are however not aware of any global recovery guarantees with this regularizer.
(Note that even though the weighted nuclear norm is linear in the singular value vector it is not a separable nor convex penalty since the singular values are non-linear functions of the matrix elements.) 
In this work we study the class of non-convex non-separable relaxations of the objective described in \eqref{eq:vectprobl} and \eqref{eq:matrixprobl}. The relaxation is obtained when replacing the regularizer with its quadratic envelope \cite{carlsson2018convex}. If $f(x) = G(\card(\x))+\|\x\|^2$ then the quadratic envelope $\reg_g$ can be defined by
\begin{equation}
\reg_g(\x) = f^{**}(\x) - \|\x\|^2,
\end{equation} 
where $f^{**}(\x)$ is the convex envelope of $f(\x)$. Thus we first add a quadratic penalty to the regularizer, then take the convex envelope and subtract the quadratic function. 
The intuition behind the choice of regularizer is that $\reg_g(\x)+\|\x-\b\|^2$ is the convex envelope of $G(\card(\x))+\|\x-\b\|^2$, see \cite{larsson-olsson-ijcv-2016}, and therefore any stationary point is a global minimizer. Under RIP the term $\|\x-\b\|^2$ is likely to behave similarly to $\|A\x-\b\|^2$ for vectors with $\card(\x) < r$. In this paper we formally study the properties of stationary points of the resulting minimizers and give conditions that guarantee global optimality of a stationary point. Note that our work exclusively deals with properties of the objective function and do not assume any particular optimization method. Any method that reaches a stationary point or a local minimum will suffice. The theory presented in this paper unifies and makes significant extensions of the results in \cite{olsson-etal-iccv-2017,olsson2017} where two special cases of the framework are studied.

\subsection{Relaxations}

In this section we give a very brief presentation of the regularizer that we use, which is taken from \cite{larsson-olsson-ijcv-2016}.
The function $G(\card(\x))$ only depends on the sorted magnitudes of the elements in $\x$, and this is also the case for $\reg_g(\x)$. We will denote these by $\btx$ and assume that $\x = D_s \pi \btx$. 
Here the vector $\s$ contains elements that are either $-1$ or $1$, $D_\s$ is a diagonal matrix with the elements of $\s$ on the diagonal and $\pi$ is a permutation matrix.
With this notation our regularizer can be written
\begin{equation}
\reg_g(\x) = \max_\btz 2\skal{\btx,\btz} - \sum_{i=1}^n \max(\tz_i^2 - g_i,0)-\|\btx\|^2,
\label{eq:zmax}
\end{equation}
see \eqref{eq:zmax} in the supplementary material for further details.
Evaluation of $\reg_g$ requires solving the maximization over $\btz$.
This can be done very fast (logarithmic time in the number of elements of $\x$). A simple (linear time) algorithm is presented in \cite{larsson-olsson-ijcv-2016}. For completeness we also present the main theory in the supplementary material (Appendix~\ref{app:convenv}).

The relaxation of \eqref{eq:vectprobl} can then be written
\begin{equation}
\reg_g (\btx) + \|A \x - \b\|^2.
\label{eq:vecrelax}
\end{equation}
The theory for the matrix case is largely identical to that of the vector case. Here the regularizer only depends on the (sorted) singular values of $X$ which we will also denote $\btx$.  The relationship between $X$ and $\btx$ is now the singular value decomposition (SVD) $X = U D_\btx V^T$, where $U$ and $V$ are orthogonal matrices. 
We will therefore write the relaxation of \eqref{eq:matrixprobl}
\begin{equation}
\reg_g (\btx) + \|\A X - \b\|^2.
\label{eq:matrelax}
\end{equation}
Note that the orthogonal matrices $U$ and $V$ are not unique if $\btx$ has elements that are zero. In the vector case we have a similar non-uniqueness in the matrix $D_\s \pi$. 
\section{Motivation}~\label{sec:motivation}

Being able to use a general prior $G$ has the benefit that we can design accurate formulations for finding the correct cardinality. While separable regularization considers the size of each variable separately the ability to apply a non-separable prior makes it possible to heavily penalize unlikely solutions. The extreme example, the so called fixed cardinality penalty
\begin{equation}
G(k) = \begin{cases}
0 & k \leq k_{\max} \\
\infty & k > k_{\max}
\end{cases}, 
\end{equation}
rules out solutions with more non-zero elements than $k_{\max}$, which cannot be achieved with a separable formulation. Less restrictive variants that regard high cardinality states unlikely (but not impossible) can also be used.

The use of a non-separable prior is not only important for modelling purposes but also effects optimization algorithms since separable formulations more often get stuck in local minima.
To understand this consider the simple one dimensional problem $|x|_0 + (x-y)^2$, where $|x|_0 = 0$ if $x=0$ and $1$ otherwise. The goal of this formulation is recover $x=y$ if $y$ is large enough not to be considered noise. By taking derivatives it is easy to see that the solution to this problem is given by $x=y$ if $|y|\geq 1$ and $x=0$ otherwise. Now suppose that $|x|_0$ is replaced by a function $f(|x|)$.
It is easy to see that the minimizer of $f(|x|)+(x-y)^2$ is either $x=0$ or a stationary point fulfilling $x = y-\frac{f'(|x|)}{2}$. Hence to really recover $x=y$ when $y$ is large enough we have to have $f'(|x|)=0$, that is, $f$ has to be constant for large values. If this is not the case $f$ will favour smaller solutions resulting in a shrinking bias. 

On the other hand a separable regularizer of the form $F(\x) = \sum_i f(x_i)$, where $f$ is constant for large values is likely to have local minima even under RIP. A 2D example depicting the shape of the regularizer is shown in Figure~\ref{fig:regex}. To the left is $\reg_g$ with $g_1 = g_2 = 1$ which is separable and obtained by applying the function to the right in Figure~\ref{fig:regex} to both coordinates.
For general dimensions these types of penalties yield regularizers that are constant for dense vectors that are large enough.  
Suppose now that $\x^*$ minimizes $\|A\x + \b\|^2$, that is $\x^*$ is a least squares solution, and there is a dense vector $\y$ in the nullspace of $A$. Then $\x^*+\lambda \y$ is also a least squares solution. If we make $\lambda$ large enough the vector $\x^*+\lambda \y$ will be located in the region where the separable relaxation is constant while at the same time minimizing the data term and therefore it is a local minimizer of the relaxation (as well as the original unrelaxed formulation). In \cite{olsson-etal-iccv-2017} conditions that guarantee uniqueness of sparse stationary points under the regularization $\reg_g$ and $g_1 = g_2 = ... =g_n$ when RIP holds were given. However, dense stationary points could not be ruled out. In our experimental section we confirm that optimization starting from a least squares solution often results in convergence to poor dense solutions, see Section~\ref{sec:experiments}. 
One way to address this problem is to accept a modest shrinking bias to make sure that the gradient of the regularizer does not vanish for large elements \cite{olsson-etal-arxiv-2018}. 
For the type of local minima described above this is likely to solve the problem, however it is not clear if there are other types of dense minima as well. 
In this work we instead consider non-separable formulations where dense vectors can be penalized harder. In the middle of Figure~\ref{fig:regex} we show $\reg_g$ with $g_1 = 1$ and $g_2 = \infty$. In the latter case the $G$ function excludes dense vectors. This is reflected in the shape of the relaxation $\reg_g$, which will clearly try to discourage cardinality 2. For vectors of cardinality $1$ (and $0$) the two options (Left and Middle) provide identical penalties. Our main results give conditions that are sufficient for guaranteeing uniqueness of stationary points among both sparse and dense vectors.
\begin{figure}[htb]
	\begin{center}
		\includegraphics[width=50mm]{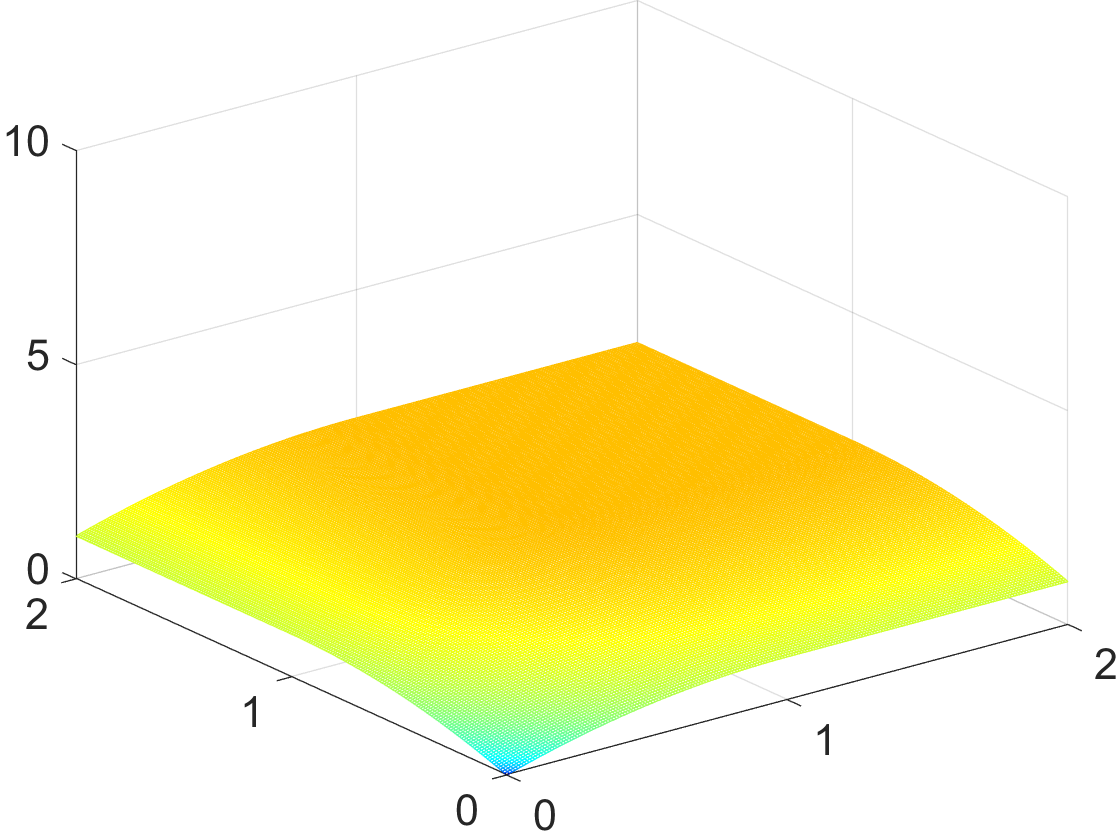}
		\includegraphics[width=50mm]{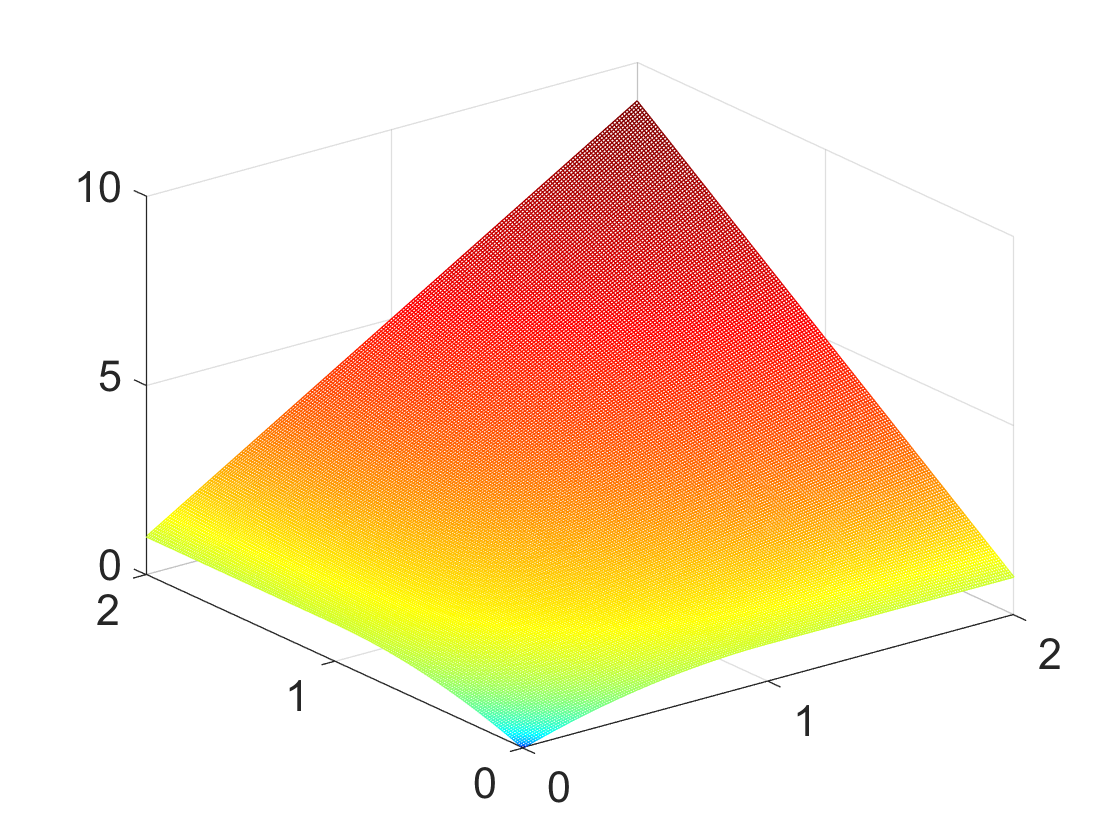}
		\includegraphics[width=50mm]{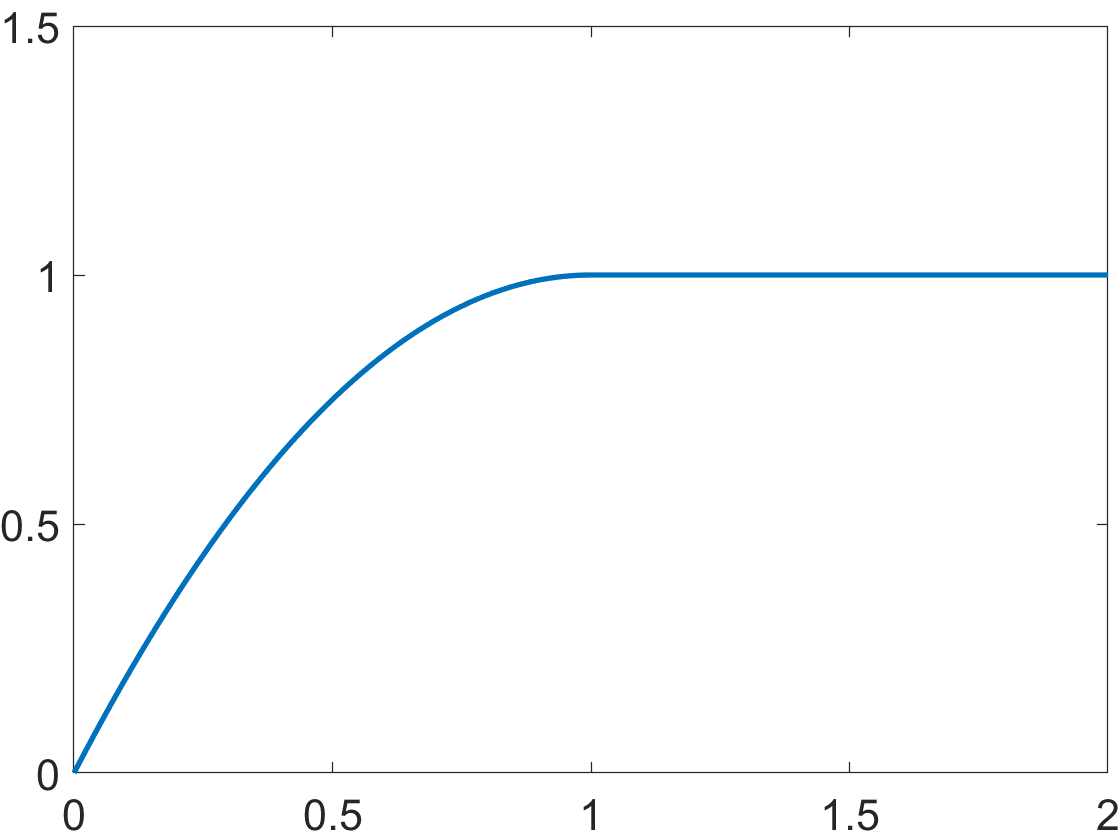}
	\end{center}
\caption{Illustrations of $\reg_g(\x)$. Left: $g_1=g_2=1$. Middle: $g_1 = 1$, $g_2=\infty$. Right: The shape of both the objectives on one of the axes.}
\label{fig:regex}
\end{figure}
\section{Theoretical Results}
In this section we will present our main theoretical results for the relaxations~\eqref{eq:vecrelax} and~\eqref{eq:matrelax}. We will state our results in terms of the vector case \eqref{eq:vecrelax}. However identical results hold for the matrix case when the magnitudes of the vector elements are replaced by the singular values of the matrix.

\subsection{What is our relaxation solving?}

The goal of our regularization is to adaptively select the appropriate rank/cardinality given the data fidelity.
If the true cardinality had been known it would be desirable solve the so called "fixed cardinality" problem
\begin{equation}
\o = \argmin_{\card(\x)\leq k} \|A\x - \b\|^2.
 \label{eq:fixedcard]}
\end{equation}
Hence we would like our formulation to determine $k$ and then to solve \eqref{eq:fixedcard]} exactly.
While many regularization methods have been proposed very few of them output solutions that are "fixed cardinality" minimizers for some $k$. As discussed in the previous section most of them add a bias that clearly favors small solutions.

The following theorem gives conditions that ensure that a particular "fixed cardinality" solution is stationary for our relaxation provided that the noise is bounded. 

\begin{theorem}\label{thm:fixedpoint}
	If $\o$ solves \eqref{eq:fixedcard]} and $\to_k > \sqrt{g_k}$ then $\o$ is a stationary point \eqref{eq:vecrelax} if
	\begin{equation}
	\|\eps\| \leq \frac{\min\{\sqrt{g_{k+1}},\to_k\}}{\|A\|},
	\end{equation}
where $\eps = A\o - b$ are the residual errors.
\end{theorem}
The proof of this result and its matrix analog is given in the supplementary material (Appendix \ref{app:fixedpointproof}).

There are two essential constraints that ensure that a fixed carnality solution $\o$ is stationary. 
Firstly, since $\to_i$ is decreasing and $\sqrt{g_i}$ is increasing with $i$ we can view the constraint $\to_k > \sqrt{g_k}$ as a threshold $\sqrt{g_i}$ which must be smaller than any non-zero element. 
The second constraint essentially states the remaining residual error that is not explained by $\o$ has to be sufficiently small. 
As a simple example we mention the noise free case $\b = A \y$ with $\card(\y)=k$. Here $\eps=0$ and therefore $\x=\y$ would be a stationary point for all choices of $g$ where the non-zero elements of $\y$ are larger than $\sqrt{g_{k}}$.

The above result does not rule out the existence of multiple stationary points. The main results of our paper are dedicated to developing conditions that ensure uniqueness of a stationary point for appropriate choices of $g$. In such cases we therefore obtain a method that is able to jointly determine the best $k$ and supply us with the corresponding "fixed cardinality" solution.

\subsection{Element Separation and Optimality}

The main result of this section will show that a sparse stationary point is under certain conditions unique.
The conditions are related to the noise and whether there is a clear truncation of the data or not.

Consider the stationary points of \eqref{eq:vecrelax}. Since $f^{**}(\x) = \reg_g(\btx)+\|\btx\|^2$ we can write the objective function \eqref{eq:vectprobl} as $f^{**}(\x) + h(\x)$, where 
\begin{equation} \begin{split}
h(\x) & = \|A\x-\b\|^2 - \|\x\|^2 \\ & = \x^T (A ^T A -I)\x - 2\x^T A^T \b + \b^T \b,
\end{split} \end{equation} 
which has
$\nabla h(\x) =  -2\z$, where $\z = (I - A^T A )\x + A^T \b$.
A point is stationary if and only if $2\z = - \nabla h(\x) \in \partial f^{**}(\x)$.
Some properties of the solutions to these equations can be understood by noting that for a fixed $\z$ the exact same equations are obtained by differentiating the objective
\begin{equation}
\reg_g(\btx) + \|\x - \z\|^2 = f^{**}(\x)-2\x^T \z + \|\z\|^2.
\label{eq:localapprox}
\end{equation} 
This expression can be seen as a local approximation (ignoring constants) around $\x$, and $\x$ is stationary in \eqref{eq:vecrelax} if and only if it is stationary in \eqref{eq:localapprox}. Furthermore, \eqref{eq:localapprox} is the convex envelope of $G(\card(\x))+ \|\x-\z\|^2$. Therefore the stationary point $\x$ is the best low cardinality approximation of $\z$ which is obtained by truncating the elements $\tz_i$ at $\sqrt{g_i}$.

It is the properties of $\z$ that decide if there could be other (sparse) stationary points or not.
Loosely speaking, our theory relies on the fact that the directional derivative of $f^{**}$ grows faster than that of $-h$. Hence if they are equal at some stationary point they cannot be so again. 
This is true if the magnitudes $\tz_i$ are well separated from their thresholds. 
If this is not the case a small change in $\tz_i$ can cause $\tx_i$ to switch from $\tx_i = \tz_i$ to $\tx_i = 0$ (or vice versa) which may result in the directional derivative of $f^{**}$ not growing sufficiently fast.
Note that for noise free recovery, that is $\b = A\y$ for some vector $\y$, we have 
$\z = (I-A^T A )\y + A^TA\y = \y$. Hence as long as the non-zero elements of $\y$ are sufficiently separated form $0$ we should be able to guarantee that this is the only sparse stationary point.  

We now state the main result:
\begin{theorem}\label{thm:statpts}
Suppose that $\x$ is a stationary point of \eqref{eq:vecrelax}, that is, $2\z \in \partial f^{**}(\x)$ with $\z = (I-A^T A)\x+A^T \b$, and the matrix $A$ fulfills \eqref{eq:vectorRIP}.
If $\card(\x)=k$, $\tx_i \notin (0,\sqrt{g_i})$ and $\btz$ fulfills 
\begin{equation}
	\tz_i \notin \left[(1-\delta_r) \sqrt{g_k},  \frac{\sqrt{g_k}}{(1-\delta_r)}\right] \text{ and }  \tz_{k+1} < (1-2\delta_r) \tz_{k}, \label{eq:singvalsep1}
\end{equation}
then any other stationary point $\x'$ has $\card(\x') > r-k$.
If in addition $k < \frac{r}{2}$ then $\x$ solves
\begin{equation}
\min_{\card(\x) < \frac{r}{2}} \reg_g(\btx)+\|A \x - \b\|^2.
\label{eq:lokmin}
\end{equation}
\end{theorem}
The proof of this theorem is given in the supplementary material (Appendix~\ref{app:main}).
To gain some more understanding of the conditions \eqref{eq:singvalsep1} we recall that the sequence $\sqrt{g_i}$ is non-decreasing, while $\btx$ and $\btz$ are non-increasing.
Therefore the first condition in \eqref{eq:singvalsep1} ensures that none of the elements $\tz_i$ are close to any of their thresholds $\sqrt{g_i}$. The second condition is to prevent a change of the permutation $\pi$, since this may result in two or more elements in $\x$ to switch from non-zero to zero and vice versa. This can happen if $\tz_k$ and $\tz_{k+1}$ are close to each other. In other cases changes in ordering does not cause changes in the support of $\x$. 

The assumption $\tx_i \notin (0,\sqrt{g_i})$ is equivalent to $f^{**}(\x) = f(\x)$, see supplementary material (Appendix~\ref{app:convenv}). 
In principle there could be stationary points where this assumption is not fulfilled if the regularizer is not strong enough to force such elements to be zero. For example, if the data term is of the form $\mu\|\x-\b\|^2$, increasing $\mu$ will eventually lead to the optimal solution being $\b$ regardless of the size of its elements.
One way to ensure that the regularization is strong enough is to require that $\|\A\| < 1$.
Then any local minimizer of \eqref{eq:vecrelax} will have $f^{**}(\x) = f(\x)$ by Theorem 4.7 in \cite{carlsson2018convex}. In addition any local minimizer of \eqref{eq:vecrelax} will be a local minimizer of \eqref{eq:vectprobl} (but not the other way around) and the global minimizers with coincide.

Before proceeding we also note that since $f^{**} \leq f$ the stationary point in Theorem~\ref{thm:statpts} will also solve
\begin{equation}
\min_{\card(\x) < \frac{r}{2}} G(\card(\x))+\|A \x - \b\|^2,
\end{equation}
if it solves \eqref{eq:lokmin}, meaning that in some sense it is the best possible sparse solution to the problem. In what follows we will therefore assume that $r=2k$.

We conclude the section by giving results that are sufficient to guarantee the existence of a stationary point fulfilling the constraints of Theorem~\ref{thm:statpts} in the presence of noise, that is, $\b=A\y+\eps$ for some clean vector $\y$. 
The following result shows that as long as the noise level is not too high there will be a stationary point fulfilling the constraints of Theorem~\ref{thm:statpts}. Moreover, this is true for a whole range of objectives as long as the thresholds $\sqrt{g}_k$ are not selected too close to the elements in $\ty$. 
Note that this result relies on worst case bounds in terms of the noise vector $\eps$.
The proof basically assumes that a single element of the stationary point $\x$ is affected by the full noise magnitude $\|\eps\|$ rather than evenly distributing the noise among the elements of $\x$. This makes the statement weaker than what can be expected in practice with for example Gaussian noise. 

\begin{theorem}\label{thm:sufficientcond}
Suppose that $\b=A\y+\eps$, for some $\y$ with $\card(\y)=k$, $\|A\| < 1$, $\delta_{2k} < \frac{1}{2}$. If
\begin{equation}
\ty_k > \frac{5}{(1-2\delta_{2k})\sqrt{1-\delta_{2k}}}\|\eps\|,
\label{eq:yconst}
\end{equation}
then \eqref{eq:vecrelax} has a stationary point $\x$, with $\card(\x)=k$, that fulfills \eqref{eq:singvalsep1}
for all choices of $G$ where 
\begin{equation}
\footnotesize
\sqrt{g_k} < (1-\delta_k)\left( \ty_k -  \frac{2\|\eps\|}{\sqrt{1-\delta_{2k}}} \right) 
\text{ and }  \sqrt{g_{k+1}} > \frac{3(1-\delta_{k})}{\sqrt{1-\delta_{2k}}}\|\eps\|.
\label{eq:gconst}
\end{equation}
\end{theorem}

Note that the proof, which is given the supplementary material (Appendix \ref{app:main2}), shows that regardless of the choice of $G$ the stationary point is always the best cardinality $k$ approximation of $\y$ (in a least squares sense). Hence if we know the cardinality beforehand we might as well use the fixed-cardinality relaxation $g_i = 0$ if $i\leq k$ and $g_i=\infty$ if $i>k$. In many practical cases the rank is not known before hand but needs to be determined through a suitably selected $g$ function. The above estimates show that a solution that is close to the original noise free vector $\y$ (and has the correct support) can often be recovered.

\subsection{Regularizers with Hard Constraints}\label{sec:hardconst}
The theory presented in the previous section shows uniqueness of sufficiently sparse stationary points, but cannot rule out dense stationary points. The main difficulty in this respect is that the RIP constraint only gives information about low cardinality vectors, and typically there are dense vectors in the nullspace of $A$. 
As illustrated in Section~\ref{sec:motivation} unbiased separable regularizers don't penalize these vectors sufficiently. 

In this section we assume that we know an upper bound $k_{\max}$ on the cardinality. This means that $g_i = \infty$ for all $i \geq k_{\max}$. 
The next result shows that relaxations resulting from such regularizers turn out to be strong enough to exclude the existence of high cardinality local minimizers and giving global optimality of a solution fulfilling the assumptions of Theorem~\ref{thm:statpts}. Note that this prior is by construction non-separable since it counts the number of non-zero element.

\begin{corollary}
Suppose that $\x$ is a stationary point of \eqref{eq:vecrelax} that fulfills the assumptions of Theorem~\ref{thm:statpts} with $r=2k$. If $\|\A\|<1$ and $g_i = \infty$ for $i \geq k$ then $\x$ is the unique local minimizer of \eqref{eq:vecrelax} and a global minimizer of \eqref{eq:vectprobl}.
\end{corollary}
\begin{proof}
According to Theorem 4.7 of \cite{carlsson2018convex} a local minimizer $\x$ of \eqref{eq:vecrelax} has $f^{**}(\x) = f(\x)$. If $\card(\x) > k$ then $g(\card(\x)) = \infty$ but $f^{**}(\x)$ is finite. 
Therefore any local minimizer has to have $\card(\x) \leq k$. 
However, according to Theorem~\ref{thm:statpts} $\x$ is the only stationary point with $\card(\x) \leq k$.	
\end{proof}

With the assumptions of Theorem~\ref{thm:sufficientcond} we get the following somewhat stronger result which also ensures existence. 
\begin{corollary}
Under the assumptions of Theorem~\ref{thm:sufficientcond} the problem \eqref{eq:vecrelax} has a unique local minimizer which is also a global minimizer of \eqref{eq:vectprobl}.
\end{corollary}

\section{Experiments}\label{sec:experiments}
In this section we illustrate the behaviour of the proposed penalty using a range of numerical experiments, both synthetic and from real applications. 
We are in particular interested in differences between the two relaxations obtained with $g_i = \mu$ for all $i$ versus $$\bar{g}_i = \begin{cases} \mu & i \leq k_{\max} \\
\infty & i > k_{\max}.
\end{cases}$$ We denote the functions obtained with these choices $G_\mu(\card)$ and $\bar{G}(\card)$ respectively and their relaxations $\reg_\mu$ and $\reg_{\bar{g}}$ respectively. Both of these attempt to estimate an unknown cardinality based on a trade-off between data fidelity and sparsity. However the second option $\reg_{\bar{g}}$ should be more robust to local minima with high cardinality as outlined in our theory.

\subsection{Synthetic data}\label{subsec:synth_data}

\subsubsection{Robustness - random matrices}\label{subsubsec:rob_rd}

To test the robustness with respect to noise we generated $10$ different problem instances - i.e.~different $100 \times 200$ Gaussian random matrices $A$ (with normalized columns), sparse real ground truths $\x_0$ and noise $\epsilon$ - for each noise level $ \|\epsilon\| / \| \b \| \in \{0.025i \, : \, i=0, \dots , 10  \}     $ and averaged the output distances $ \| \tilde{\x} - \x_0 \|/ \|\x_0 \|$, with $\tilde{\x} $ approximated solution computed by the minimizing algorithm. Each $\x_0 $ was chosen with random cardinality between $10$ and $18$ and with the property that $ \min_i |\x_{0,i}|> 2 \sqrt{2}$. To approximately recover $x_0$ we use the formulation $g_i =2 $ for $i \le k_{\text{max}} =20$ and $g_i=\infty$ otherwise, and refer to it as $\reg_{\bar{g}}$ in the results. We used \(\mathbf{0}\) as starting point and the Forward-Backward Splitting as minimizing algorithm. 

For comparison we also test the Least Absolute Shrinkage and Selection Operator (LASSO), $\ell^p$ with $p=1/2, 2/3$, and the Smoothly Clipped Absolute Deviation (SCAD) which are popular approaches from the literature.
To avoid issues with parameter selection and achieve the best possible performance for these competing methods their parameters were picked using a line search \emph{for each problem instance}. So, for each problem instance $A$, $\epsilon$ and $\x_0$, we computed $ \| \tilde{\x} - \x_0 \|/ \|\x_0 \|$ for several choices of the involved parameters and stored the best outcome. The result is shown in the top left graph in Figure~\ref{fig:distvsnoise}. Here $\reg_g$ outperforms the other relaxations consistently giving the best fit to the ground truth data.
The behaviour that all methods approach the ground truth solution when the noise decreases is due to the fact that the parameter is exactly tuned to the noise level for each problem. We emphasize that in real applications such a strategy is normally not feasible when the noise level is unknown. 

A more realistic scenario is to use the same parameter setting for all noise levels. Therefore we also performed a second batch of experiments where one single parameter for each competitor method was used. The parameters chosen minimize the average error through all the noise levels and all trials. The outcome (second row of Figure \ref{fig:distvsnoise}) highlights the benefits of the noise-invariance of the bias-free methods ($\reg_g$ and SCAD). These can handle varying noise levels with a single parameter setting.

\subsubsection{Robustness - concatenation between Fourier transform and identity}\label{subsubsec:rob_fi}
 For a matrix \( A \in \mathbb{C}^{m \times N} \), the quantity \[  \mu(A) = \max_{1 \le i \ne j \le N} | \langle \a _i , \a _j \rangle |  \] is known as \emph{mutual coherence} of \(A\). In compressive sensing a small mutual coherence is desirable because it controls all the restricted isometry constants from above (cfr.~Prop.~6.2 in \cite{Rauhut_Foucart}); the matrix \( A = [F|I] \) with \(F\) being the \(1D\) Fourier transform matrix and \(I\) the identity matrix, is known to have mutual coherence \( = 1/\sqrt{m}\) and it is often used in compressive sensing algorithms or techniques benchmarking.
 
 We ran a similar set of tests as in Section \ref{subsubsec:rob_rd} using $A$ of dimensions $100 \times 200 $ instead of random matrices: we generated $10$ different problem instances - i.e.~different sparse real ground truths $\x_0$ and noise $\epsilon$ - for each noise level $ \|\epsilon\| / \| \b \| \in \{0.025i \, : \, i=0, \dots , 10  \}     $ and averaged the output distances $ \| \tilde{\x} - \x_0 \|/ \|\x_0 \|$, with $\tilde{\x} $ approximated solution computed by the minimizing algorithm. Each $\x_0 $ was chosen with cardinality $10$ and with the property that $ \min_i |\x_{0,i}|> 2 \sqrt{2}$. In our formulation we selected $g_i = 2$ for $i=1, \dots , k_{\text{max}}=20$; the competitor methods are the same as in Section \ref{subsubsec:rob_rd} and their parameters choices were again made via the same technique(s). The outcome mostly mirrors that in Section \ref{subsubsec:rob_rd}, see the second columns of Figure \ref{fig:distvsnoise}.

\begin{figure*}
    \centering
    \def\w{85mm}
    \begin{tabular}{cc}
    Random matrices, line-searched parameters & \( [F|I] \), line-searched parameters  \\
    \includegraphics[width=\w]{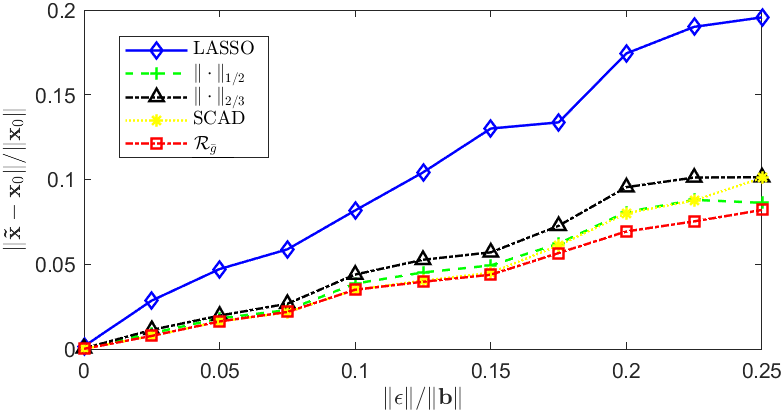} &
    \includegraphics[width=\w]{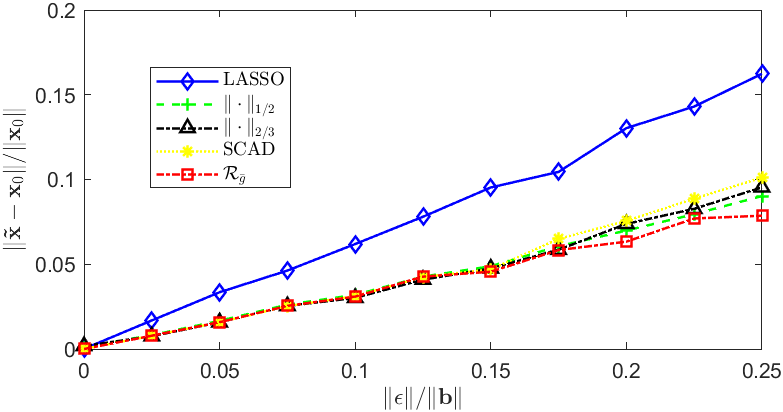} \\
    Random matrices, fixed parameters & \( [F|I] \), fixed parameters  \\
    \includegraphics[width=\w]{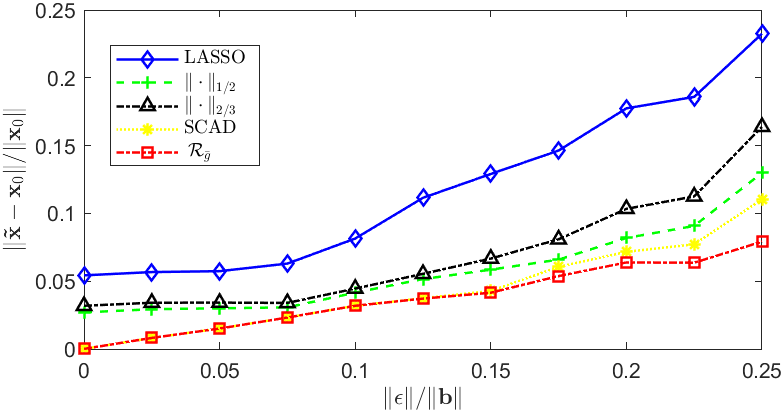} &
    \includegraphics[width=\w]{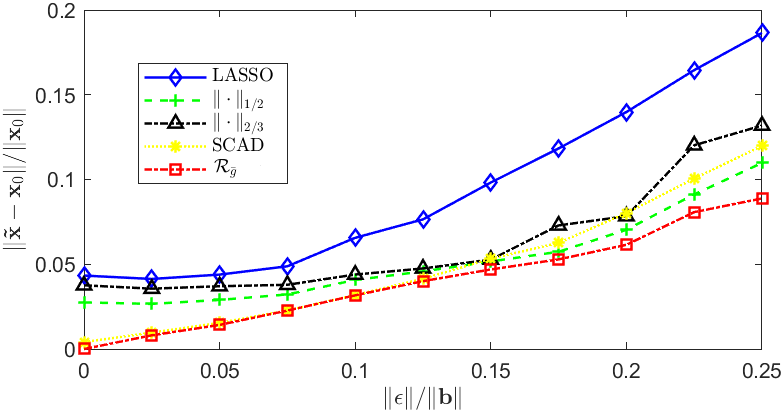}
    
    \end{tabular}
    \caption{Normalized error vs normalized noise level. The first column experiments with random matrices, see Section \ref{subsubsec:rob_rd}; the second, concatenation of Fourier matrix with identity, see Section \ref{subsubsec:rob_fi}.}
    \label{fig:distvsnoise}
\end{figure*}

\subsubsection{Sparsity}\label{subsubsec:sparse}
In a second batch of experiments we studied the cardinality of the retrieved approximation when a starting point not too far from the ground truth is employed. For a fixed noise level $ \|\epsilon\| / \| \b \| = 0.15$ we fixed a triplet $A$ (Gaussian, with normalized columns), $\epsilon$ and $\x_0$ with $\card{(\x_0)}=10 $ and again $ \min_i |\x_{0,i}|> 2 \sqrt{2} $, and we generated $250$ different random (with uniform distribution) starting points $\x_S $ with $ 0.2\|\x_0\| \le \|\x_S\| \le 3\|\x_0\| $. In Table \ref{table:75x100randst} we display mean and standard deviation of $ S_m(\tilde{\x},\x_0)= \card{(\supp{(\tilde{\x}) } \, \triangle \, \supp{(\x_0) } )} $\footnote{$\triangle$ is the set-theoretic symmetric difference.} and normalized distance to ground truth $ \| \tilde{\x} - \x_0 \|/ \|\x_0 \|$ in the scenario $75 \times 200$. The parameter choice for the competing methods was made with a line search for each single problem instance (as in the previous section). To achieve a the correct support as well as a good fit to the ground truth solution we selected the parameter that minimized the quantity $0.8 \, \card{(\supp{(\tilde{\x}) } \, \triangle \, \supp{(\x_0) } )} / \|\x_0\|_0 + 0.2\| \tilde{\x} - \x_0 \|/ \|\x_0 \|$. The proposed regularizer (indicated in the legend as ``$\reg_{\bar{g}}$") displays a solid behaviour.

\begin{table}[h!]
\centering
\begin{minipage}[t]{0.48\linewidth}
\caption{Random starting point, $75 \times 200$ scenario}
\begin{tabular}{|c|cc|cc|}
\hline
\multirow{2}{*}{} & \multicolumn{2}{c|}{$ \| \tilde{\x} - \x_0 \|/\|\x_0\| $} & \multicolumn{2}{c|}{$S_m(\tilde{\x},\x_0)$} \\ 
& \textbf{Mean} & \textbf{St.~dev.} & \textbf{Mean} & \textbf{St.~dev.}  \\ \hline $\mathcal{R}_{\bar{g}}$ & 0.0768 & $1.96 \cdot 10^{-7}$ & 0 & 0 \\ \hline
$\mathcal{R}_\mu$   & 0.0768  & $2.20 \cdot 10^{-6}$   & 0   & 0 \\ \hline
SCAD & 0.0898 & $2.03 \cdot 10^{-7}$ & 0 & 0 \\ \hline
 LASSO  & 0.7290 & $1.09 \cdot 10^{-7}$ & 2 & 0 \\ \hline
$\ell^{1/2}$ & 0.1465 & 0.1626 & 0.4840 & 1.6915 \\ \hline
$\ell^{2/3}$ & 0.1147 & 0.0052 & 0 & 0 \\
\hline
\end{tabular}
\label{table:75x100randst}
\end{minipage}
\begin{minipage}[t]{0.48\linewidth}
\centering
\caption{$A^\dagger \b$ as starting point, $100 \times 200$ scenario}

\begin{tabular}{|c|cc|cc|}
\hline
\multirow{2}{*}{} & \multicolumn{2}{c|}{$ \| \tilde{\x} - \x_0 \|/\|\x_0\| $} & \multicolumn{2}{c|}{$S_m(\tilde{\x},\x_0)$} \\ 
& \textbf{Mean} & \textbf{St.~dev.} & \textbf{Mean} & \textbf{St.~dev.}  \\ \hline $\mathcal{R}_{\bar{g}}$ & 0.0643 & 0.1079 & 0.1080 & 0.3273 \\ \hline
$\mathcal{R}_\mu$   & 0.0636  & 0.1343   & 0.2360   & 2.9772 \\ \hline
SCAD & 0.0885 & 0.0582 & 0.1240 & 0.3758 \\ \hline
 LASSO  & 0.3863 & 0.1490 & 1.2800 & 1.5734 \\ \hline
$\ell^{1/2}$ & 0.0856 & 0.1449 & 0.4560 & 2.4658 \\ \hline
$\ell^{2/3}$ & 0.0650 & 0.0157 & 0 & 0 \\
\hline
\end{tabular}
\label{table:100x200lsqst}
\end{minipage}
\end{table}
\begin{table}[h!]
\centering
\caption{Local minima counter}
\begin{tabular}{||c c c c c||} 
 \hline
  &$\mathcal{R}_{\bar{g}}$ & $\mathcal{R}_\mu$ & $\|\cdot\|_{1/2}$ & $\|\cdot\|_{2/3} $\\ [0.5ex] 
 \hline\hline
 Loc.~min.~detected & 0 & 100 & 1 & 0 \\ 
  [1ex] 
 \hline
\end{tabular}

\label{table:locmin}

\end{table}
\begin{figure*}
    \centering
    \def\w{42mm}
    \begin{tabular}{cccc}
    Drink & Pickup & Stretch & Yoga \\
    \includegraphics[width=\w]{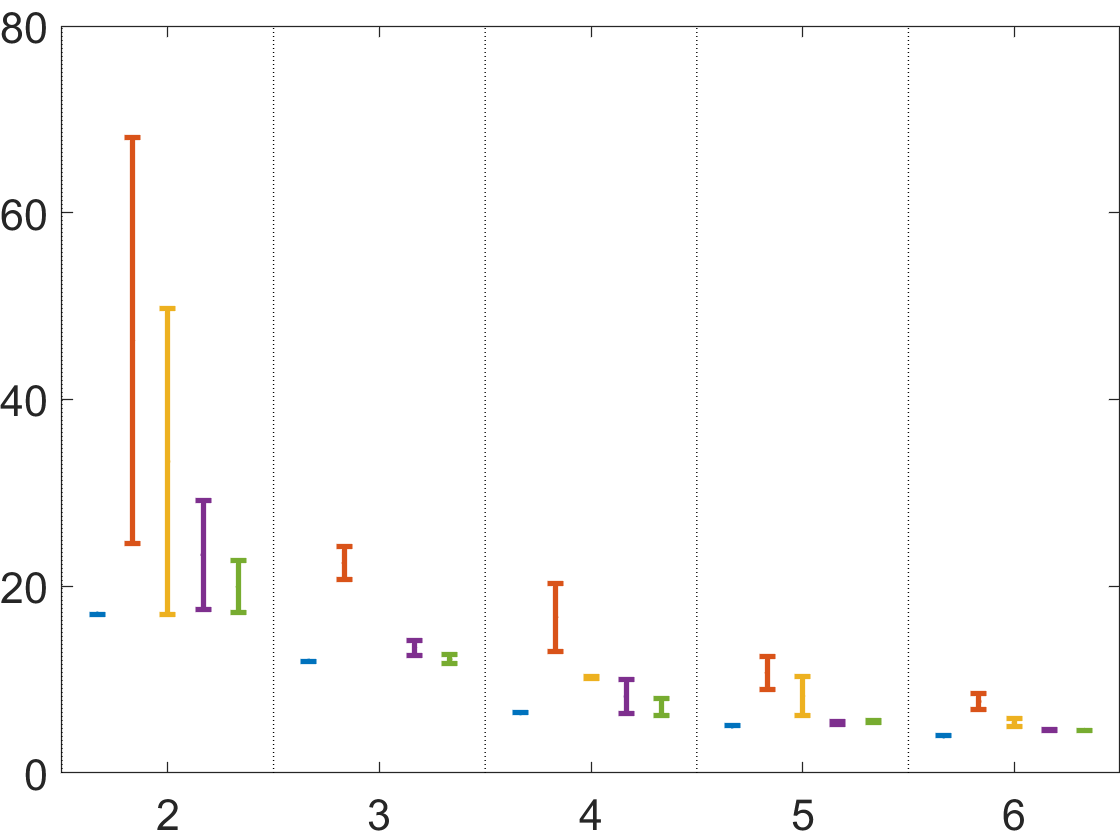} &
    \includegraphics[width=\w]{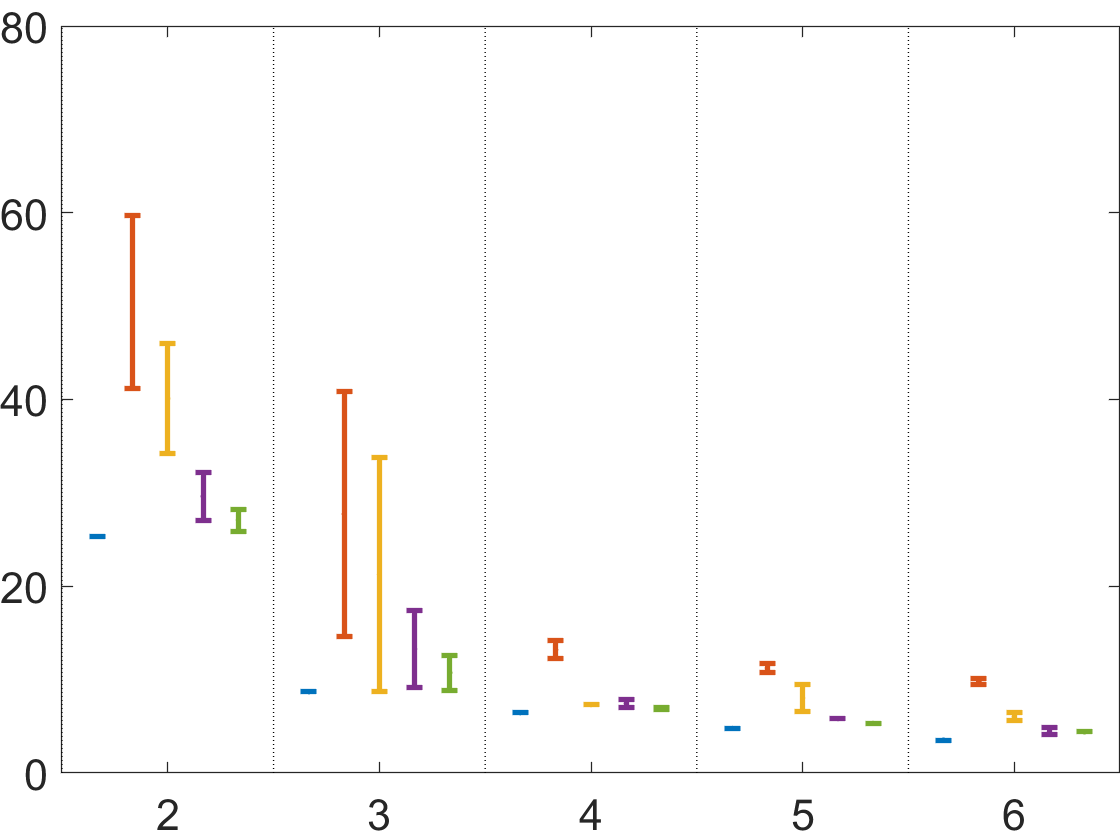} &
    \includegraphics[width=\w]{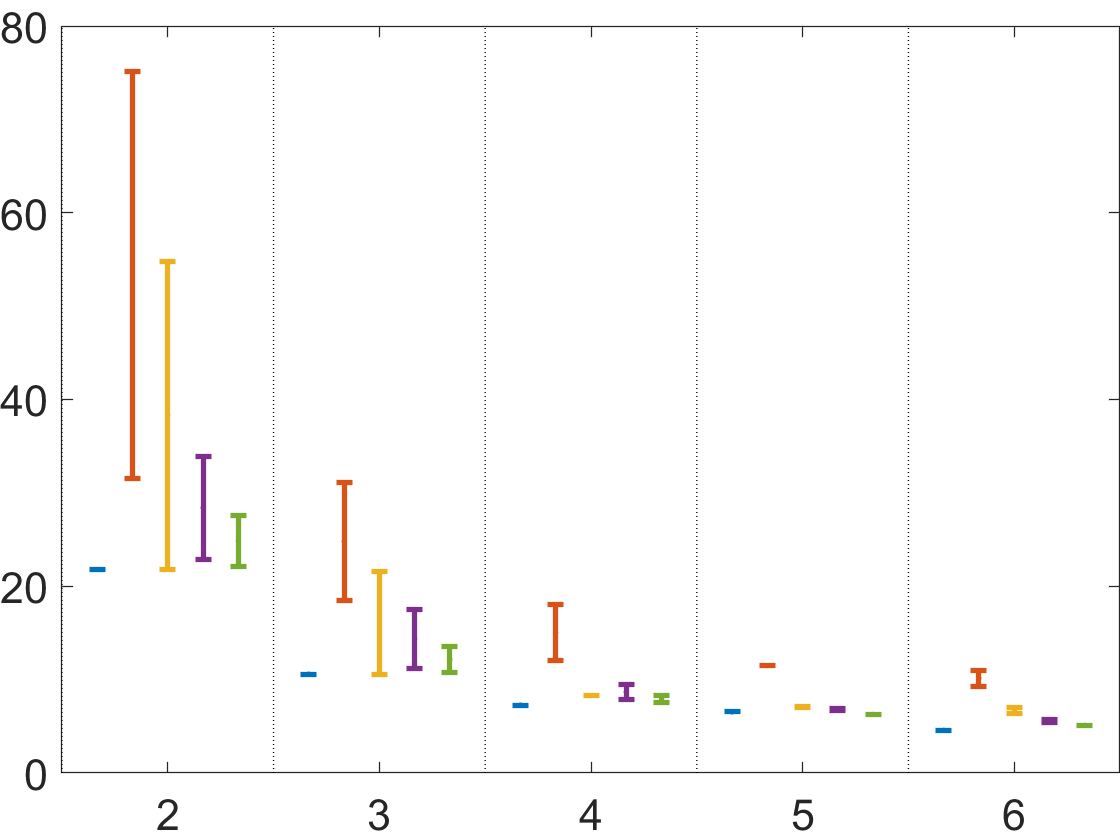} &
    \includegraphics[width=\w]{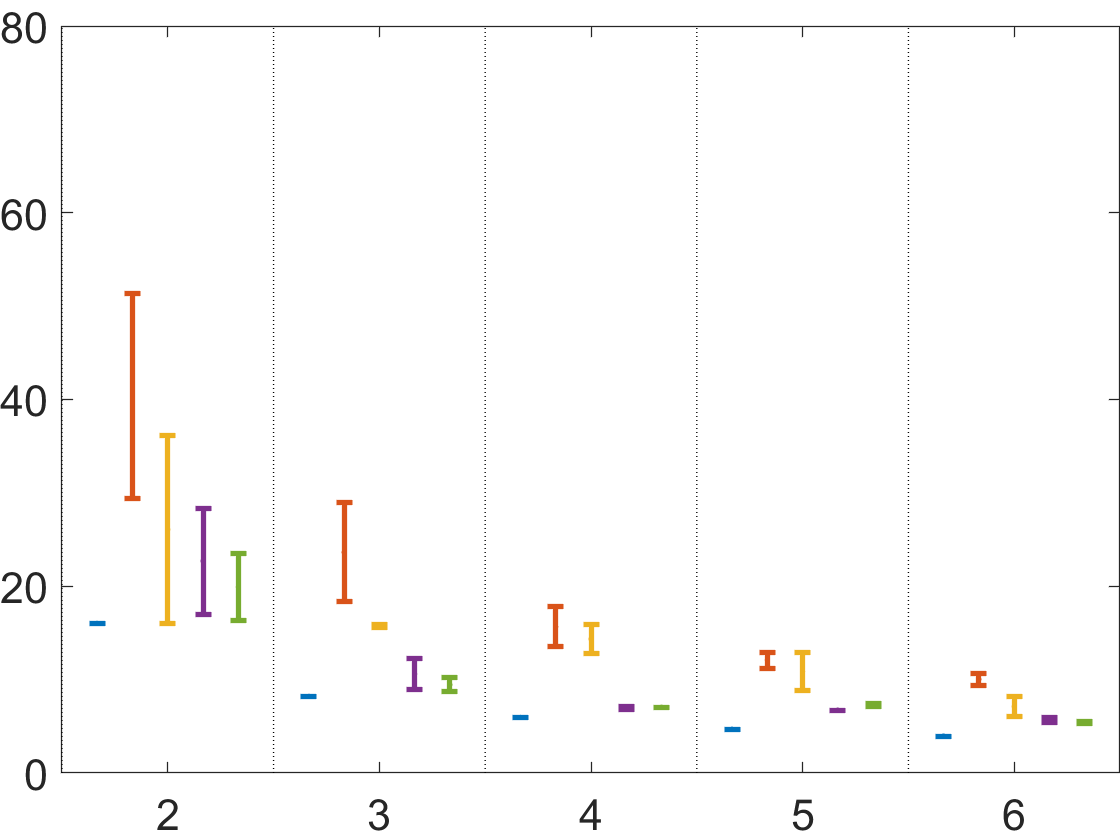}
    \end{tabular}
    \caption{Rank of $X^\#$ (x-axis) versus data fit $\|RX-M\|_F$ (y-axis) for the MOCAP sequences, drink, pickup, stretch and yoga, used in \cite{dai-etal-ijcv-2014}. Blue - $\reg_g$, red - $\|\cdot\|_*$, yellow - SCAD, purple - Schatten-$2/3$ and green - Schatten-$1/2$.}
    \label{fig:mocap}
\end{figure*}

\subsubsection{Local minima suppression}\label{subsubsec:min_suppr}
In \cite{carlssonIP_2020} was numerically shown that the algorithm minimizing the functional $\reg_\mu(\x) + \|A \x - \b\|^2$ tends to get stuck in high cardinality local minima when the least square solution is used as starting point. The present paper can be also seen as an attempt to overcome that issue and in this section we numerically confirm that those high cardinality local minima seem to be suppressed when our new penalty is employed. We again generated $250$ different problem instances with the same specs as in Section \ref{subsubsec:sparse} and used a least square solution to the linear system $A \x = \b $ as starting point for the algorithm. Note that since $A$ has a nullspace there are in general multiple least squares solutions. In Table we \ref{table:100x200lsqst} used $A^\dagger \b$. The results show that while $A^\dagger \b$ seems like a sensible starting point it often gives sub-optimal results. 
This is in particular true for the bias free regularizer $\reg_\mu$ that has difficulty recovering from a high cardinality starting point. The proposed $\reg_{\bar{g}}$ is in general much less affected than the other bias free methods. We also remark that strictly speaking deviations from the ground truth $x_0$ may not be a result of local minima since $x_0$ is not a minimizer.

The point $A^\dagger \b$ is an intuitive initialization and indeed it is a bit surprising that it produces local minima. A less intuitive choice is what is described in Section~\ref{sec:motivation}, that is, points of type $A^\dagger \b + \x_k $ where $\x_k \in \text{ker}(A) $ is dense such that $\min_i |(A^\dagger \b + \x_k)_i|> 2 \sqrt{2} $, which we consider in Table~\ref{table:locmin}. These are still least square solutions and they are located in the region where the penalty $\reg_\mu(\x) $ is constant; thus they are local minima for the functional $\reg_\mu(\x) + \|A \x - \b\|^2$. For a $100 \times 200$ random matrix $A$ we generated $100$ linearly independent $\x_k$ and tested whether the points $A^\dagger \b + \x_k$ are local minima or not for the functionals $\reg_\mu(\x) + \|A \x - \b\|^2$, \eqref{eq:vecrelax} and $ \|\x\|_p + \|A \x - \b\|^2$ (for $p=1/2, 2/3$); we here picked $k_{\text{max}}=16 > \card(\x_0) \in \{9, 10,\dots,14\} $ chosen randomly. Table \ref{table:locmin} displays the results of our experiment: it shows that all those points are (as expected) local minima for $\reg_\mu(\x) + \|A \x - \b\|^2$ but not for \eqref{eq:vecrelax} and motivates one more time the constructions in the present manuscript.

\subsection{Rank Regularization for NRSfM}
We conclude our experiments by considering an application of the matrix version of our framework.
Non-rigid structure from motion (NRSfM) is a classical computer vision problem where object dynamics is modeled using a rank constraint. In this section we follow \cite{dai-etal-ijcv-2014} and extract a deforming model from point tracks obtained with a moving camera. Under the linear deformation assumption \cite{bregler-etal-cvpr-2000,dai-etal-ijcv-2014} the deforming 3D point cloud can be represented using a low rank matrix $S$ where row $i$ contains $x-$, $y-$ and $z-$ coordinates of the point cloud when image $i$ was captured. To recover $S$ we solve
\begin{equation}
\min_S R_g(S) + \|R S^\# - M\|^2.     
\end{equation}
Here $R$ is a matrix containing camera rotations and $S^\#$ is a matrix where the elements of $S$ have been reordered, see \cite{dai-etal-ijcv-2014} for detailed definitions. 

In Figure~\ref{fig:mocap} we use data from \cite{dai-etal-ijcv-2014} to test our regularizer $\reg_g$, with $g_k=0$ if $k \leq k_{\max}$ and $\infty$ otherwise, against the nuclear norm, SCAD and the Schatten-norms. Only $\reg_g$ can directly penalize the rank. The other competing methods have different parameters that needs to be tuned to indirectly obtain a certain rank. Not that while a small parameter change may not change the rank it can still change the solution since the bias is affected. For a fair comparison we therefore sample parameters over a whole range of values to see what data fit can we achieve with settings that give particular rank.
Note that since the operator $S \mapsto R S^\#$ has low rank matrices in its null-space it therefore does not fulfill RIP, therefore our method could have local minima.
For ranks between $2$ an $6$ Figure~\ref{fig:mocap} shows error bars covering the best and the worst data fit for each method. The bias is most clearly visible for the lowest rank (2) where the methods have to suppress more noise. Our method (blue) consistently gives the lowest data fit for each rank. 

\section{Conclusions}
In this paper we have presented and analysed a general framework for sparsity and rank regularization of linear least squares problems. Our regularizers are bias free an non-separable which admits increased modeling power compared to standard separable version. Our theoretical analysis shows that under the RIP constraint stationary points are often unique even though our framework is non-convex. Our empirical results further demonstrate that we outperform competing methods in terms of accuracy and robustness.

\newpage
{\small
\bibliographystyle{plain}
\bibliography{newlib}

\begin{thebibliography}{10}

\bibitem{argyriou-etal-nips-2012}
Andreas Argyriou, Rina Foygel, and Nathan Srebro.
\newblock Sparse prediction with the k-support norm.
\newblock In F.~Pereira, C.~J.~C. Burges, L.~Bottou, and K.~Q. Weinberger,
  editors, {\em Advances in Neural Information Processing Systems 25}, pages
  1457--1465. Curran Associates, Inc., 2012.

\bibitem{basri-etal-ijcv-2007}
Ronen Basri, David Jacobs, and Ira Kemelmacher.
\newblock Photometric stereo with general, unknown lighting.
\newblock {\em International Journal of Computer Vision}, 72(3):239--257, May
  2007.

\bibitem{blumensath2009iterative}
Thomas Blumensath and Mike~E. Davies.
\newblock Iterative hard thresholding for compressed sensing.
\newblock {\em Applied and computational harmonic analysis}, 27(3):265--274,
  2009.

\bibitem{bredies2015minimization}
Kristian Bredies, Dirk A.~Lorenz ~, and Stefan Reiterer.
\newblock Minimization of non-smooth, non-convex functionals by iterative
  thresholding.
\newblock {\em Journal of Optimization Theory and Applications},
  165(1):78--112, 2015.

\bibitem{bregler-etal-cvpr-2000}
C.~Bregler, A.~Hertzmann, and H.~Biermann.
\newblock Recovering non-rigid 3d shape from image streams.
\newblock In {\em The IEEE Conference on Computer Vision and Pattern
  Recognition (CVPR)}, 2000.

\bibitem{candes-etal-acm-2011}
Emmanuel~J. Cand\`{e}s, Xiaodong Li, Yi~Ma, and John Wright.
\newblock Robust principal component analysis?
\newblock {\em J. ACM}, 58(3):11:1--11:37, 2011.

\bibitem{candes2009exact}
Emmanuel~J. Cand\`es and Benjamin Recht.
\newblock Exact matrix completion via convex optimization.
\newblock {\em Foundations of Computational Mathematics}, 9(6):717--772, 2009.

\bibitem{candes-tao-2006}
Emmanuel~J. Cand\`es and Terence Tao.
\newblock Near-optimal signal recovery from random projections: Universal
  encoding strategies?
\newblock {\em IEEE transactions on information theory}, 52(12):5406--5425,
  2006.

\bibitem{carlsson2018convex}
Marcus Carlsson.
\newblock On convex envelopes and regularization of non-convex functionals
  without moving global minima.
\newblock {\em Journal of Optimization Theory and Applications, to appear},
  2019.

\bibitem{olsson-etal-arxiv-2018}
Marcus Carlsson, Daniele Gerosa, and Carl Olsson.
\newblock Bias reduction in compressed sensing, 2018.

\bibitem{carlsson2019biased}
Marcus Carlsson, Daniele Gerosa, and Carl Olsson.
\newblock An un-biased approach to low rank recovery.
\newblock {\em arXiv preprint arXiv:1909.13363}, 2019.

\bibitem{carlssonIP_2020}
Marcus Carlsson, Daniele Gerosa, and Carl Olsson.
\newblock An unbiased approach to compressed sensing.
\newblock {\em Inverse Problems}, 36 115014, 2020.

\bibitem{chartrand2007exact}
Rick Chartrand.
\newblock Exact reconstruction of sparse signals via nonconvex minimization.
\newblock {\em IEEE Signal Processing Letters}, 14(10):707--710, 2007.

\bibitem{dai-etal-ijcv-2014}
Yuchao Dai, Hongdong Li, and Mingyi He.
\newblock A simple prior-free method for non-rigid structure-from-motion
  factorization.
\newblock {\em International Journal of Computer Vision}, 107(2):101--122,
  2014.

\bibitem{donoho-elad-2002}
David~L. Donoho and Michael Elad.
\newblock Optimally sparse representation in general (non-orthogonal)
  dictionaries via l1-minimization.
\newblock In {\em PROC. NATL ACAD. SCI. USA 100 2197–202}, 2002.

\bibitem{eriksson-etal-cvpr-2015}
Anders Eriksson, Trung Thanh~Pham, Tat-Jun Chin, and Ian Reid.
\newblock The k-support norm and convex envelopes of cardinality and rank.
\newblock In {\em IEEE Conference on Computer Vision and Pattern Recognition
  (CVPR)}, June 2015.

\bibitem{fan2001variable}
Jianqing Fan and Runze Li.
\newblock Variable selection via nonconcave penalized likelihood and its oracle
  properties.
\newblock {\em Journal of the American Statistical Association},
  96(456):1348--1360, 2001.

\bibitem{Rauhut_Foucart}
Simon Foucart and Holger Rauhut.
\newblock A mathematical introduction to compressive sensing.
\newblock 2013.

\bibitem{garg-etal-cvpr-2013}
Ravi Garg, Anastasios Roussos, and Lourdes Agapito.
\newblock Dense variational reconstruction of non-rigid surfaces from monocular
  video.
\newblock In {\em The IEEE Conference on Computer Vision and Pattern
  Recognition (CVPR)}, 2013.

\bibitem{garg-etal-ijcv-2013}
Ravi Garg, Anastasios Roussos, and Lourdes Agapito.
\newblock A variational approach to video registration with subspace
  constraints.
\newblock {\em International Journal of Computer Vision}, 104(3):286--314,
  2013.

\bibitem{gillis-glineur-siam-2011}
N.~Gillis and F.~Glinuer.
\newblock Low-rank matrix approximation with weights or missing data is
  np-hard.
\newblock {\em SIAM Journal on Matrix Analysis and Applications}, 32(4), 2011.

\bibitem{grussler-etal-tac-2018}
Christian Grussler, Anders Rantzer, and Pontus Giselsson.
\newblock Low-rank optimization with convex constraints.
\newblock {\em IEEE Transactions on Automatic Control}, 63(11):4000--4007,
  2018.

\bibitem{gu-2016}
Shuhang Gu, Qi~Xie, Deyu Meng, Wangmeng Zuo, Xiangchu Feng, and Lei Zhang.
\newblock Weighted nuclear norm minimization and its applications to low level
  vision.
\newblock {\em International Journal of Computer Vision}, 121, 07 2016.

\bibitem{hu-etal-pami-2013}
Yao Hu, Debing Zhang, Jieping Ye, Xuelong Li, and Xiaofei He.
\newblock Fast and accurate matrix completion via truncated nuclear norm
  regularization.
\newblock {\em IEEE Transactions on Pattern Analysis and Machine Intelligence},
  35(9):2117--2130, 2013.

\bibitem{iglesias2020accurate}
Jose~Pedro Iglesias, Carl Olsson, and Marcus~Valtonen Örnhag.
\newblock Accurate optimization of weighted nuclear norm for non-rigid
  structure from motion, 2020.

\bibitem{kumar-arxiv-2019}
Suryansh Kumar.
\newblock A simple prior-free method for non-rigid structure-from-motion
  factorization : Revisited.
\newblock {\em CoRR}, abs/1902.10274, 2019.

\bibitem{larsson-olsson-ijcv-2016}
Viktor Larsson and Carl Olsson.
\newblock Convex low rank approximation.
\newblock {\em International Journal of Computer Vision}, 120(2):194--214,
  2016.

\bibitem{loh2013regularized}
Po-Ling Loh and Martin~J. Wainwright.
\newblock Regularized m-estimators with nonconvexity: Statistical and
  algorithmic theory for local optima.
\newblock In {\em Advances in Neural Information Processing Systems}, pages
  476--484, 2013.

\bibitem{loh-wainwright-arxiv-2014}
Po{-}Ling Loh and Martin~J. Wainwright.
\newblock Support recovery without incoherence: {A} case for nonconvex
  regularization.
\newblock {\em arXiv preprint}, arXiv:1412.5632, 2014.

\bibitem{loh2017support}
Po-Ling Loh and Martin~J Wainwright.
\newblock Support recovery without incoherence: A case for nonconvex
  regularization.
\newblock {\em The Annals of Statistics}, 45(6):2455--2482, 2017.

\bibitem{mazumder2011sparsenet}
Rahul Mazumder, Jerome~H. Friedman, and Trevor Hastie.
\newblock Sparsenet: Coordinate descent with nonconvex penalties.
\newblock {\em Journal of the American Statistical Association},
  106(495):1125--1138, 2011.

\bibitem{mcdonald-etal-jmlr-2016}
Andrew~M. McDonald, Massimiliano Pontil, and Dimitris Stamos.
\newblock New perspectives on k-support and cluster norms.
\newblock {\em J. Mach. Learn. Res.}, 17(1):5376–5413, January 2016.

\bibitem{mohan2010iterative}
Karthik Mohan and Maryam Fazel.
\newblock Iterative reweighted least squares for matrix rank minimization.
\newblock In {\em Annual Allerton Conference on Communication, Control, and
  Computing}, pages 653--661, 2010.

\bibitem{natarajan1995sparse}
Balas~K. Natarajan.
\newblock Sparse approximate solutions to linear systems.
\newblock {\em SIAM journal on computing}, 24(2):227--234, 1995.

\bibitem{oh-etal-pami-2016}
Tae-Hyun Oh, Yu-Wing Tai, Jean-Charles Bazin, Hyeongwoo Kim, and In~S. Kweon.
\newblock Partial sum minimization of singular values in robust pca: Algorithm
  and applications.
\newblock {\em IEEE Transactions on Pattern Analysis and Machine Intelligence},
  38(4):744--758, 2016.

\bibitem{olsson-etal-iccv-2017}
Carl Olsson, Marcus Carlsson, Fredrik Andersson, and Viktor Larsson.
\newblock Non-convex rank/sparsity regularization and local minima.
\newblock {\em Proceedings of the International Conference on Computer Vision},
  2017.

\bibitem{olsson2017}
Carl Olsson, Marcus Carlsson, and Erik Bylow.
\newblock A non-convex relaxation for fixed-rank approximation.
\newblock In {\em 2017 IEEE International Conference on Computer Vision
  Workshops (ICCVW)}, pages 1809--1817, Oct 2017.

\bibitem{olsson-etal-cvpr-2010}
Carl Olsson, Anders Eriksson, and Richard Hartley.
\newblock Outlier removal using duality.
\newblock In {\em IEEE Int. Conference on Computer Vision and Pattern
  Recognition}, pages 1450--1457, 2010.

\bibitem{oymak-etal-2015}
Samet Oymak, Amin Jalali, Maryam Fazel, Yonina~C. Eldar, and Babak Hassibi.
\newblock Simultaneously structured models with application to sparse and
  low-rank matrices.
\newblock {\em IEEE Transactions on Information Theory}, 61(5):2886--2908,
  2015.

\bibitem{oymak2011simplified}
Samet Oymak, Karthik Mohan, Maryam Fazel, and Babak Hassibi.
\newblock A simplified approach to recovery conditions for low rank matrices.
\newblock In {\em IEEE International Symposium on Information Theory
  Proceedings (ISIT)}, pages 2318--2322, 2011.

\bibitem{pan2015relaxed}
Zheng Pan and Changshui Zhang.
\newblock Relaxed sparse eigenvalue conditions for sparse estimation via
  non-convex regularized regression.
\newblock {\em Pattern Recognition}, 48(1):231--243, 2015.

\bibitem{recht-etal-siam-2010}
Benjamin Recht, Maryam Fazel, and Pablo~A. Parrilo.
\newblock Guaranteed minimum-rank solutions of linear matrix equations via
  nuclear norm minimization.
\newblock {\em SIAM Rev.}, 52(3):471--501, August 2010.

\bibitem{tomasi-kanade-ijcv-1992}
Carlo Tomasi and Takeo Kanade.
\newblock Shape and motion from image streams under orthography: A
  factorization method.
\newblock {\em International Journal of Computer Vision}, 9(2):137--154, 1992.

\bibitem{tropp-2006}
Joel~A. Tropp.
\newblock Just relax: Convex programming methods for identifying sparse signals
  in noise.
\newblock {\em IEEE transactions on information theory}, 52(3):1030--1051,
  2006.

\bibitem{tropp-2015}
Joel~A. Tropp.
\newblock Convex recovery of a structured signal from independent random linear
  measurements.
\newblock In {\em Sampling Theory, a Renaissance}, pages 67--101. 2015.

\bibitem{wright-etal-pami-2009}
John Wright, Allen~Y. Yang, Arvind Ganesh, S.~Shankar Sastry, and Yi~Ma.
\newblock Robust face recognition via sparse representation.
\newblock {\em IEEE Trans. Pattern Anal. Mach. Intell.}, 31(2):210–227,
  February 2009.

\bibitem{yan-pollefeys-pami-2008}
Jingyu Yan and Marc Pollefeys.
\newblock A factorization-based approach for articulated nonrigid shape, motion
  and kinematic chain recovery from video.
\newblock {\em IEEE Trans. Pattern Anal. Mach. Intell.}, 30(5):865--877, 2008.

\bibitem{zhang2010nearly}
Cun-Hui Zhang.
\newblock Nearly unbiased variable selection under minimax concave penalty.
\newblock {\em The Annals of Statistics}, 38(2):894--942, 2010.

\bibitem{zhang2012general}
Cun-Hui Zhang and Tong Zhang.
\newblock A general theory of concave regularization for high-dimensional
  sparse estimation problems.
\newblock {\em Statistical Science}, pages 576--593, 2012.

\bibitem{zou2008one}
Hui Zou and Runze Li.
\newblock One-step sparse estimates in nonconcave penalized likelihood models.
\newblock {\em Annals of statistics}, 36(4):1509, 2008.

\end{thebibliography}
}

\newpage

\appendix

\section{Preliminaries on Convex Envelopes and Subdifferentials}~\label{app:convenv}
In this section we review some basic facts about convex envelopes and their sub-differentials that we will use in our theory.
Throughout the section we will assume that any infimum is attained. This is true for example if the function is lower semi continuous with bounded level sets, which is the case for our objective function $f(\x) = g(\card(\x))+\|\x\|^2$. In addition the quadratic term grows faster than any linear term of the type $\skal{\x,\y}$ and therefore this is also true when we add linear terms.

The convex envelope $f^{**}$ of a function $f$ is the largest convex function that fulfills $f^{**} \leq f$. For a convex function we should have $f^{**}(\sum_i \lambda_i \x^j) \leq  \sum_j \lambda_j f^{**}( \x^j)$, $0 \leq \lambda_j$, $\sum_j \lambda_j=1$.
At a point $\x = \sum_j \lambda_j \x^j$ where $f(\x) > \sum_j \lambda_i f( \x^j)$ we compute the value $f^{**}(\x)$ by minimizing over convex combinations of points using 
\begin{equation}
f^{**}(x)= \min \left\{ \sum_{j=1}^{d+1} \lambda_j f(\x^j); \quad \sum_{j=1}^{d+1} \lambda_j \x^j = \x, \quad \sum_{j=1}^{d+1} \lambda_j = 1, \quad \lambda_j > 0  \right\}.
\label{eq:convenv}
\end{equation}
It can be shown (using Carathéodory's Theorem) that it is enough to consider combinations of $d+1$ points if $\x \in \mathbb{R}^d$. Figure~\ref{fig:convenv} shows one example of convex envelope. Here the two functions $f^{**}$ and $f$ coincide at $x=0$ and $|x| > 1$.
If $x \in (0,1)$ where the functions differ and the value of $f^{**}$ is computed using the convex combination $(1-x)f(0)+xf(1)$. Note that when $f$ and $f^{**}$ differs the function $f^{**}$ will be affine in some direction.
\begin{figure}[htb]
\begin{center}
\includegraphics[width=80mm]{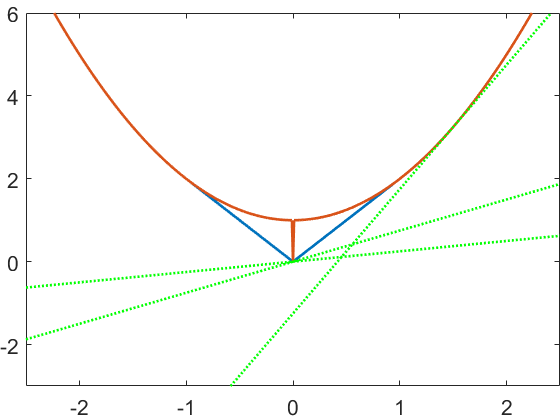}
\end{center}
\caption{An example of convex envelope. Here  $f^{**}(x) = \mu - \max(\sqrt{\mu}-|x|,0)^2+x^2$ (blue curve) and $f(x) = \mu \card(x)+x^2$ (orange curve).
Green dotted lines are supporting hyperplanes of the form $xy - f^*(y)$ for tree different $y$.}\label{fig:convenv}
\end{figure}

An alternative way of computing $f^{**}$ is using supporting hyperplanes and the conjugate function
\begin{equation}
f^*(\y) = \max_\x \skal{\x,\y}-f(\x).
\label{eq:conjugate}
\end{equation}
From the definition it is clear that 
\begin{equation}
f^*(\y) \geq \skal{\x,\y}-f(\x),
\end{equation}
for all $\x$. Rearranging terms we get
\begin{equation}
f(\x) \geq \skal{\x,\y}-f^*(\y),
\end{equation}
which is a an affine function in $\x$ and therefore a supporting hyperplane to $f$.
Figure~\ref{fig:convenv} shows three supporting hyperplanes for $f$. 
Note that these touch $f^{**}$ in (at least) one point. 
For each $\x$ we can find a hyperplane that touches $f^{**}(\x)$ which means that
\begin{equation}
f^{**}(\x) = \max_\y \skal{\x,\y} - f^*(\y),
\end{equation}
that is, the convex envelope is the conjugate of the conjugate function.

For a convex function $f^{**}$ the set of sub-gradients $\partial f^{**}(\x)$ at a point $\x$ is defined as all vectors $\y$ such that
\begin{equation}
f^{**}(\x') - f^{**}(\x) \geq \skal{\y,\x'-\x}, \quad \forall \x'.
\end{equation}
or equivalently
\begin{equation}
\skal{\y,\x}-f^{**}(\x) \geq \skal{\y,\x'}-f^{**}(\x') , \quad \forall \x'.
\end{equation}
Since we clearly have equality when $\x'=\x$ this means that $ \skal{\y,\x}-f^{**}(\x) = \max_{\x'}\skal{\y,\x'}-f^{**}(\x') = f^{***}(\x) = f^{*}(\x)$.
Rearranging terms shows that
\begin{equation}
\skal{\y,\x}-f^{*}(\y) = f^{**}(\x) = \max_{\y'}\skal{\y',\x}-f^{*}(\y').
\end{equation}
Thus the set of sub-gradients at a point $\x$ are all the vectors $\y$ that achieves the maximal value in the second conjugation. In points where $f^{**}$ is non-differentiable the function has several sub-gradients. In a differentiable point the only sub-gradient is the standard gradient.

The following result does not appear to be standard but is crucial for our main theorem. Therefore we state it somewhat more formally below.

\begin{lemma}\label{conj:decomp}
	Suppose that for a point $\x$ we have $f(\x) > f^{**}(\x)$.
	Then there is a set of points $\{\x^j\}$ such that 
	\begin{equation} 
	\x = \sum_j \lambda_j \x^j, \quad 0 \leq \lambda_j \leq 1, \quad \sum_j \lambda_j = 1,
	\end{equation}
	$f^{**}(\x^j)=f(\x^j)$ and 
	\begin{equation}
	f^{**}(\x) = \sum_j \lambda_j f(\x^j).
	\end{equation}
	In addition $\partial f^{**}(\x) \subset \bigcap_j \partial f^{**}(\x^j)$.
\end{lemma}

\begin{proof}
Consider the convex combination $\x=\sum_j \x^j$ that solves the minimization in \eqref{eq:convenv}.

We have that $f(\x^j) \geq f^{**}(\x^j)$.
Assume further that $f(\x^j) > f^{**}(\x^j)$ for some $j$. Then we have 
\begin{equation}
f^{**}(\sum_j \lambda_j \x^j) =f^{**}(\x) = \sum_j \lambda_j f(\x^j) > \sum_j \lambda_j f^{**}(\x^j),
\end{equation}
which contradicts the convexity of $f^{**}$. Therefore $f(\x^j) = f^{**}(\x^j)$ for all $j$. 

Now consider a subgradient $\y\in \partial f^{**}(\x)$. By definition we have that
\begin{equation}
f^{**}(\x') \geq f^{**}(\x) + \skal{\y,\x'-\x}.
\end{equation}
Now assume that 
\begin{equation}
f^{**}(\x^j) > f^{**}(\x) + \skal{\y,\x^j-\x},
\end{equation}
for some j.
Then we have
\begin{equation}
\underbrace{\sum_j \lambda_j f^{**}(\x^j)}_{= f^{**}(\x)} > \underbrace{\sum_j \lambda_j f^{**}(\x)}_{=f^{**}(\x)} + \underbrace{\sum_j \lambda_j \skal{\y,\x^j-\x}}_{=0},  
\end{equation}
which shows that we must have 
\begin{equation}
f^{**}(\x^j) = f^{**}(\x) + \skal{\y,\x^j-\x}.
\end{equation}
This gives us 
\begin{equation}
f^{**}(\x^j) + \skal{\y,\x'-\x^j} =  f^{**}(\x) + \skal{\y,\x'-\x} \leq f^{**}(\x'), 
\end{equation}
which shows that $\y \in \partial f^{**}(\x^j)$ for all $j$.
\end{proof}

\subsection{The Conjugate of $f$}
We now consider our class of functions $f(\x) = g(\card(\btx))+\|\btx\|^2$.
Consider the conjugate \eqref{eq:conjugate}. Since $f(\x)$ only depends on $\btx$ and not the signs or ordering of the elements it is clear that the elements of $\x$ should have the same sign and ordering as those in $\y$ to maximize the term $\skal{\x,\y}$.
Therefore we get
\begin{equation}
f^*(\y) = \max_\btx \skal{\btx,\bty}-g(\card(\btx))-\|\btx\|^2.
\end{equation} 
This can equivalently be written
\begin{equation}
\max_k \max_{\|\btx\|_0 = k} \skal{\btx,\bty} - \sum_{i=1}^k \left(g_i + \tx_i^2 \right).
\end{equation}
Completing squares gives
\begin{equation}
\max_k \max_{\|\btx\|_0 = k} -\|\btx-\frac{1}{2}\bty\|^2 + \frac{1}{4}\|\bty\|^2 - \sum_{i=1}^k g_i.
\end{equation}
It is clear that the inner maximization is solved by letting $\tx_i = \frac{\ty_i}{2}$ if $i\leq k$ and $\tx_i= 0$ otherwise. After some simple manipulations this gives the conjugate function
\begin{equation}
f^*(\y) = \sum_{i=1}^n \max(\frac{1}{4}\ty_i^2 - g_i,0).
\end{equation}
Not that the computations for the matrix case are close to identical. 
In this case we maximize the scalar product $\skal{X,Y}$ 
when $X$ and $Y$ SVDs with the same U and V matrices (von Neumann's trace theorem),
in this case $\skal{X,Y}=\skal{\btx,\bty}$.

\subsection{The biconjugate of $f$}
Taking the conjugate once more gives
\begin{equation}
f^{**}(\x) = \max_\y \skal{\x,\y} - \sum_{i=1}^n \max(\frac{1}{4}\ty_i^2 - g_i,0).
\end{equation}
Again the second term only depends on the elements of $\bty$ and therefore
\begin{equation}
f^{**}(\x) = \max_\bty \skal{\btx,\bty} - \sum_{i=1}^n \max(\frac{1}{4}\ty_i^2 - g_i,0).
\end{equation}
For ease of notation we let $\bty = 2\btz$ which gives
\begin{equation}
f^{**}(\x) = \max_\btz 2\skal{\btx,\btz} - \sum_{i=1}^n \max(\tz_i^2 - g_i,0).
\label{eq:zmax}
\end{equation}
The maximization over $\btz$ does in general not have any closed form solution but has to be evaluated numerically. Note however that it is a concave maximization problem that we can solve efficiently. One exception where we can find $\btz$ is for points $\x$ where $f^{**}(\x)=f(\x)$. In what follows we will derive some properties of the maximizing $\z$ that simplifies the optimization. 

We first consider the elements of $\btz$ independently without regard for their ordering.
Each one has an objective function of the form
\begin{equation}
c_i(\tz_i) := 2\tx_i \tz_i - \max(\tz^2_i - g_i,0). 
\label{eq:singvalfun}
\end{equation}
\begin{figure}
\begin{center}
\includegraphics[width=27mm]{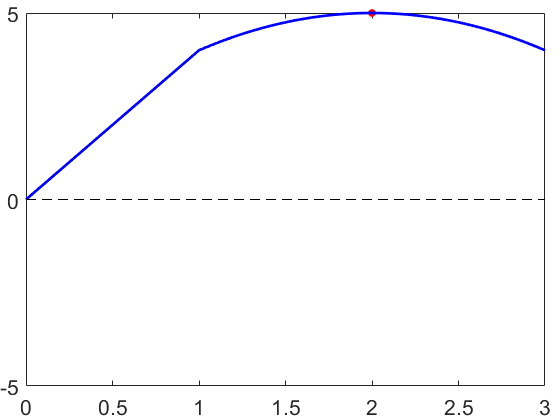}
\includegraphics[width=27mm]{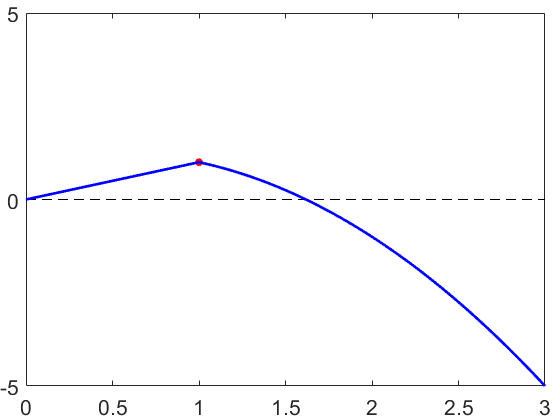}
\includegraphics[width=27mm]{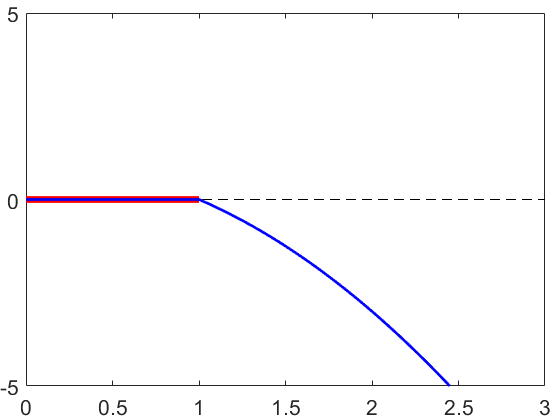}
\end{center}
\caption{The objective function \eqref{eq:singvalfun} for $g_i=1$ and $x_i = 0, \frac{1}{2} $ and $2$. The maximizing points are shown in red in each case.}
\label{fig:singvalfun}
\end{figure}
Figure~\ref{fig:singvalfun} shows $c_i(\tz_i)$ for different values of $\tx_i$. When $\tx_i \geq \sqrt{g}_i \geq 0$ there is a unique maximizing point in $\tz_i=\tx_i$.
If $0 < \tx_i \leq \sqrt{g}_i$ the maximizing point is $\tz_i=\sqrt{g_i}$.
In the last case where $0=\tx_i \leq \sqrt{g_i}$ any $\tz_i \in [0,\sqrt{g_i}]$ is a maximizer. 
Suppose that we select $k$ such that 
\begin{eqnarray}
\tx_i \geq \sqrt{g_i} & i \leq k\\
\tx_i < \sqrt{g_i} & i > k.\\
\end{eqnarray}
Then the unconstrained minimizers $u_i^*$ of \eqref{eq:singvalfun} can be written
\begin{equation}
u_i^* \in \begin{cases}
\tx_i &  i \leq k \\
\sqrt{g_i} & i > k, \tx_i \neq 0 \\
[0, \sqrt{g}_i] & i > k, \tx_i = 0 
\end{cases}
\end{equation}
Before we proceed any further we note that if $\tx_i \notin (0,\sqrt{g_i})$ the second case will not occur. We can then select the elements of $\btz$ so that each $\tz_i$ maximizes $c_i(\tz_i)$ without violating the ordering constraint.
\begin{lemma}\label{lemma:knownz}
If $\tx_i \notin (0,\sqrt{g_i})$ then the vectors maximizing \eqref{eq:zmax} are given by
\begin{equation}
\tz^*_i = \begin{cases}
\tx_i & i \leq k \\
s_i & i > k \\
\end{cases},
\label{eq:redsform}
\end{equation}
where $s_i$, $i=k+1,...,n$ is non-increasing and $s_i \in [0,\min(\tx_k,\sqrt{g_{k+1}})]$.
In this case we also have 
\begin{equation}
\begin{split}f^{**}(\x) = \sum_{i=1}^k c_i(\tx_i) & = \sum_{i=1}^k 2\tx_i \tz_i - \max(\tx_i^2 -g_i,0) \\ & =  \sum_{i=1}^k g_i+\tx_i^2 = f(\x). \end{split}
\end{equation}
\end{lemma}

Before we proceed to the general case we note that if $\tx_i \in (0,\sqrt{g_i})$ for some $i$ then 
\begin{equation}
c_i(u^*_i) = \begin{cases}
\tx_i^2 + g_i & i \leq k \\
2\tx_i\sqrt{g_i} & i > k, \tx_i \neq 0 \\
0 & i > k, \tx_i = 0 
\end{cases}.
\end{equation}
Since $2 \tx_i\sqrt{g_i} < \tx_i^2 + g_i$ if $x_i < \sqrt{g_i}$ it is clear that this implies that
\begin{equation}
f^{**}(\x) < \sum_{i=1}^{\card(\x)} g_i + \tx_i^2 = f(\x) .
\end{equation}

For the general case the unconstrained minimizers are not non-increasing.
To handle this we consider the best value $\tz_i^*$ given values for its neighbors $\tz^*_{i-1}$ and $\tz^*_{i+1}$. 
It is clear from the figures above that if the unconstrained minimizer $u_i^*$ is unique then the best choice is of $\tz^*_i$ is 
\begin{equation}
\tz^*_i = 
\begin{cases}
\tz^*_{i+1} &  u^*_i \leq \tz^*_{i+1} \\
u^*_{i} & u^*_i \in [\tz^*_{i+1},\tz^*_{i-1}] \\
\tz^*_{i-1} & u^*_i \geq \tz^*_{i-1} \\
\end{cases}.
\end{equation}
Here we have adopted the convention that $\tz_0 = \infty$ and $\tz_{n+1} = 0$.
In the non-unique case we similarly have that $\tz_i^* \in [0,\sqrt{g_i}] \cap [\tz^*_{i+1},\tz^*_{i-1}]$ if this intersection is non-empty or $\tz_i^* = \tz_{i+1}^*$.

\begin{lemma} Suppose that $\{u^*_i\}$ is not monotone for all $i$ such that $\tx_i = 0$.
	Let $p$ be defined so that the sequence $\{u^*_i\}$ is non-increasing for $i \leq p$ and non-decreasing for $i > p$, whenever $\tx_i\neq 0$.
The constrained maximizers $\tz_i^*$ will then fulfill
\begin{equation}
\tz^*_i \in \begin{cases}
\max(u^*_i,\tz^*_{i+1}) & i \leq p \\
\tz^*_{i+1} & i > p,\tx_i \neq 0 \\
[0,\min(u^*_i,\tz^*_{i-1})] & i > p, \tx_i = 0 \\
\end{cases}.
\label{eq:zrec}
\end{equation}
\end{lemma}
\begin{proof}
We first consider $i > k$ with $\tx_i = 0$. 
Since $\tx_i$ is non-increasing it is clear we can make $c_i(\tz_i) = ...= c_n(\tz_n) = 0$ by letting $\tz_i = ... = \tz_{n} = 0$, regardless of what $\tz_{i-1}$ is. Any optimal solution therefore has to have $c_i(\tz_i) = ...= c_n(\tz_n) = 0$,
which is achieved when $\tz_i^* \in [0,\min(u^*_i,\tz^*_{i-1})]$.

Since $\tz_0^* = \infty$ we have that $u_1^* \notin [\tz^*_{0},\tz_{2}^*]$ if and only if $\tz^*_{2} > u_1^*$. Therefore it is clear that $\tz^*_1 = \max(u_1^*,\tz^*_2)$. 
Now suppose $\tz^*_{i-1} = \max(u_{i-1}^*,\tz^*_{i})$ for $i \leq p$ then 
$\tz^*_{i-1} \geq u_{i-1}^* \geq u_i^*$, which means that either 
$u^*_i \in [\tz^*_{i+1},\tz^*_{i-1}]$, in which case $\tz^*_i = u_i^*$, or $u^*_i \leq \tz^*_{i+1}$ which gives $\tz_i^* = \tz_{i+1}^*$. This proves the first case in \eqref{eq:zrec}.

Now suppose that $i\geq p+1$ has $\tx_i \neq 0$. If $\tz_i^* \leq u_i^*$ then $\tz^*_{i} \leq u^*_{i+1}$ since $\{u_i^*\}$ is not decreasing and therefore $\tz^*_{i+1} = \tz^*_{i}$ according to \eqref{eq:zrec}. If $\tz^*_i > u_i^*$ we again have $\tz^*_{i+1} = \tz^*_{i}$ according to \eqref{eq:zrec}.
\end{proof}

\begin{corollary}\label{cor:minimizers}
The constrained minimizers $\tz_i^*$ can be written
\begin{equation}
\tz^*_i = \begin{cases}
\max(u^*_i,s) & i \leq p \\
s & i > p,\tx_i \neq 0 \\
s_i & i > p, \tx_i = 0 \\
\end{cases}.
\label{eq:sform}
\end{equation}
Here $u^*_p \leq s$, $\{s_i\}$, $i=k+1,...,n$ is non-increasing and $s_i \in [0,\min(\sqrt{g_i},s)]$.
\end{corollary}
\begin{proof}
The first two cases in \eqref{eq:sform} are fairly obvious. First it is clear that $s := \tz^*_{p} = \tz^*_{i} $, for all $i>p$ with $\tx_i \neq 0$, which is the middle case in \eqref{eq:sform}. 
Next we see that $\tz^*_{p-1} = \max(u^*_{p-1},\tz^*_k) = \max(u^*_{p-1},\max(u^*_{p},s)) = \max(u^*_{k-1},s)$ since $u^*_{p-1} \geq u^*_{p}$.
Repeating the same argument again shows the first case in \eqref{eq:sform}.

Finally we note that $s_i$ non increasing and $s_i \in [0,\min(\sqrt{g_i},s)]$ implies that $s_i \leq s$ and therefore also that
$s_i \in [0,\min(\sqrt{g_i},s)]\cap [0,s_{i-1}] = [0,\min(\sqrt{g_i},s_{i-1})]$, which shows the third case of \eqref{eq:sform}. 
 
To see that $u^*_p \leq s$ we note that all residuals $c_i$ are non-decreasing with $s$ when $s < u^*$.
\end{proof}

\section{Proof of Theorem~\ref{thm:fixedpoint}}\label{app:fixedpointproof}
In this section we give the proof of Theorem~\ref{thm:fixedpoint} which shows that "fixed cardinality/rank" solutions are stationary in our relaxation~\eqref{eq:vecrelax}. 
The proofs for vector and matrix cases are somewhat different and therefore we treat them separately.
\subsection{The vector case}

\begin{proof}[Proof of Theorem~\ref{thm:fixedpoint}]
The objective function of \eqref{eq:vecrelax} can be written $f^{**}(\x)+h(\x)$ where $h(\x)= - \|x\|^2+\|Ax-b\|^2$.
A stationary point therefore fulfills $-\nabla h(x) \in \partial f^{**}(x)$.
We have $\nabla h(\x) = -2\x + 2 A^T(A \x - \b)$ which yields
\begin{equation}
-A^T (A \x - \b) = \z-\x,	
\end{equation}
where $2\z \in \partial f^{**}(\x)$.
Now suppose that $\o$ fulfills the requirements of the theorem and let S be the set of nonzero elements of $\o$.
The sub-differential $\partial f^{**}(\o)$ consists of the maximizing $\z$-vectors given in Lemma~\ref{lemma:knownz}.
The the vector $\z-\o$ is zero for every element in $S$. To see that the same is true for $A^T(A \o - \b)$ we note that 
\begin{equation}
	\o = \argmin_{\supp(\x)=S} \|A\x - \b\|^2 = \argmin_{\supp(\x)=S} \|A_S \x - \b\|^2,
\end{equation}
where $A_S$ is constructed by taking $A$ and setting the columns no in $S$ to zero. 
Therefore the normal equations $A_S^T (A_S \o - \b)= A_S^T(A \o - \b)=0$ hold which shows that the elements of $A^T(A\o-\b)$ that are in $S$ all vanish.

It now remains to show that the elements in the complement of $S$ are smaller than $\min\{\tx_k,\sqrt{g_{k+1}}\}$. This is however clear since by assumption
\begin{equation}
	\|A^T (A^T\o - b)\| \leq \|A\| \|\eps\| \leq \min\{\tx_k,\sqrt{g_{k+1}}\}.
\end{equation}
We remark that estimating the size of the elements by the vector norm is a simple but very crude estimation and the result is therefore likely to hold under much more generous conditions. 
\end{proof}

\subsection{The matrix case}

\begin{proof}[Proof of Theorem~\ref{thm:fixedpoint}.]
Similar to the vector case we need to show that 
\begin{equation}
	-\A^*  (\A \O - \b) = Z-\O,
\end{equation}
where $2 Z \in \partial f^{**}(\O)$. The matrix $Z$ is in the sub differential of $f^{**}(\O)$ if we can find orthogonal matrices $U$ and $V$ such that $\O = U D_\bto V^T$ and $Z = U D_\btz V^T$. Here $\bto$ and $\btz$ are the singular values of $\O$ and $Z$ respectively. The matrices $D_\bto$ and $D_\btz$ are diagonal matrices with elements $\bto$ and $\btz$.
Note that $\O$ is typically of low rank $D_\bto$ and $D_\z$ can be partitioned into block matrices
\begin{equation}
	D_\bto = \begin{bmatrix}
		\Sigma & 0 \\
		0 & 0
	\end{bmatrix} 
\text{ and }
	D_\btz= \begin{bmatrix}
		\Sigma & 0 \\
		0 & \Delta
	\end{bmatrix}. 
\end{equation}
Here $\Sigma$ contains the $k$ non-zero singular values of $\O$. Due to Lemma~\ref{lemma:knownz} the $D_\btz$ also contains this block. The matrix $\Delta$ contains the singular values of $\btz$ that correspond to zeros in $\bto$.
We can make a corresponding partition of the $U$ and $V$ matrices into 
\begin{equation}
U = \begin{bmatrix}
	\bar{U} & \bar{U}_\perp
\end{bmatrix}
\text{ and }
V = \begin{bmatrix}
	\bar{V} & \bar{V}_\perp
\end{bmatrix},
\end{equation}
where $\bar{U}$ and $\bar{V}$ are the first $k$ columns of $U$ and $V$ respectively.
Note that only $\bar{U}$ and $\bar{V}$ are uniquely determined by $\O$. The matrices $\bar{U}_\perp$ and $\bar{V}_\perp$ can be selected arbitrarily as long as they are orthogonal to $\bar{U}$ and $\bar{V}$ respectively.
Any choice of $\bar{U}$, $\bar{V}$ and $\Delta$, where the elements of $\Delta$ are less than $\min\{\to_k,\sqrt{g_{k+1}}\}$ gives us a $Z$ that is in the sub differential. 
Consequently we have 
\begin{equation}
	Z - \O = \begin{bmatrix}
		\bar{U} & \bar{U}_\perp 
	\end{bmatrix}
	\begin{bmatrix}
		0 & 0 \\
		0 & \Delta
	\end{bmatrix}
\begin{bmatrix}
	\bar{V}^T \\ \bar{V}^T_\perp
\end{bmatrix} = \bar{U}_\perp \Delta \bar{V}^T_\perp.
\end{equation}
\end{proof}
We now consider term $\A^*  (\A \O - \b)$. We have
\begin{equation} \begin{split}
\|\A(\O+tH) - \b\|^2  = &  t^2\|\A H\|^2 + 2t\skal{H, \A^* (\A \O - \b)} \\ & + \|\A \O - \b\|^2. \end{split}
\end{equation}
Recall that $\O$ minimizes the left hand side over all matrices with rank at most $k$. Since the linear term dominates the quadratic one for small $t$ we must have
\begin{equation}
	\skal{H, \A^* (\A \O - \b)} \geq 0
\end{equation}
for all $H$ such that $\rank(\O+tH) \leq k$. Since $\Sigma$ has full rank it is clear that any matrix of the form
\begin{equation}
	H = \begin{bmatrix}
		\bar{U} & \bar{U}_\perp 
	\end{bmatrix}
	\begin{bmatrix}
		H_{11} & H_{12} \\
		H_{21} & 0
	\end{bmatrix}
	\begin{bmatrix}
		\bar{V}^T \\ \bar{V}^T_\perp
	\end{bmatrix}
\end{equation}
fulfills this requirement. It is now easy to see that 
\begin{equation}
 -\A^* (\A \O - \b) = U_\perp M V_\perp^T,
\end{equation}
where $M$ is some matrix. Furthermore since $U_\perp$ and $V_\perp$ can be selected freely (as long as they are perpendicular to $U$ and $V$ respectively) we can assume that $M$ is diagonal. What remains is therefore to estimate its singular values, which similarly to the vector case is done by
\begin{equation}
	\|\A^*(\A O - \b)\|_2 \leq \|\A\| \|\eps\| \leq \min\{\to_k,\sqrt{g_{k+1}}\}.
\end{equation}

\section{Proof of Theorem~\ref{thm:statpts}}\label{app:main}

In this section we prove our main theorem. 
The proof requires a growth estimate of the subgradients of $f^{**}$ which we give in the following lemmas.
\begin{lemma}
If $\z \in \partial f^{**}(\x)$ and $\z' \in \partial f^{**}(\x')$ and $d\leq 1$ then 
\begin{equation}
\skal{\z'-\z,\x'- \x} > d\|\x'-\x\|^2,
\label{eq:skal1} 
\end{equation}
if
\begin{equation}
\skal{\pi \btz'-\btz,\pi\btx'- \btx} > d\|\pi\btx'-\btx\|^2,
\end{equation}
for all permutation matrices $\pi$.
\end{lemma}
\begin{proof}
	We have that \eqref{eq:skal1} can be written
\begin{equation}
C-\skal{\z'-d \x', \x} - \skal{\z-d \x,\x'} > 0,
\label{eq:Cineq}
\end{equation}	 
where
\begin{equation}
C = \skal{\z'-d\x',\x'} + \skal{\z - d \x,\x}.
\end{equation}
Note that since the elements of $\z'$ and $\x'$ have the same signs and $\tz'_i \geq \tx'_i$ for all $i$ the term $C$ is independent of signs.
For fixed magnitudes and permutations the term $\skal{\z'-d \x', \x} + \skal{\z-d \x,\x'}$ is clearly maximized when $\z'$ and $\z$ have the same signs. In which case we have
\begin{equation}
\skal{\z'-\z,\x'- \x} = \skal{\pi \btz'-\btz,\pi\btx'- \btx}
\end{equation}
and $\|\x'-\x\|^2 = \|\pi\btx'-\btx\|^2$ for some permutation $\pi$.
\end{proof}

In the matrix case we have $Z \in \partial f^{**}(X)$ and $Z' \in \partial f^{**}(X')$. Recall that here $\btx$, $\btz$, $\btx'$, $\btz'$, are the singular values of the matrices $X$,$Z$,$X'$,$Z'$ respectively. The corresponding statement is then that
\begin{equation}
\skal{Z'-Z,X'- X} > d\|X'-X\|^2
\end{equation}
holds whenever \eqref{eq:Cineq} holds. 
The proof is however more complicated than the vector case. We therefore refer the reader to Proposition 4.5 \cite{carlsson2019biased} from which it is clear that the above statement holds.

We are now ready to establish the growth estimates on the directional derivatives needed to prove Theorem~\ref{thm:statpts}. We will first consider directional derivatives between points where the relaxation is tight, that is $f_g(\x) = f^{**}_g(\x)$. In the subsequent result we then relax this assumption to only be valid for one of the points (namely the stationary point we want to prove is unique). 

\begin{lemma}\label{lemma:subgradbnd}
Suppose that $2\z \in \partial f^{**}(\x)$ and $2\z' \in \partial f^{**}(\x')$, and that neither $\btx$ nor $\btx'$ have values in $(0,\sqrt{g_i})$. If the elements of $\btz$ fulfill
\begin{equation}
\tz_i \notin \left[(1-d) \sqrt{g_k},  \frac{\sqrt{g_k}}{(1-d)}\right] \text{ and }  \tz_{k+1} < (1-2d) \tz_{k}, \label{eq:singvalsep}
\end{equation}
where $k$ is defined so that $\tx_i \geq \sqrt{g_i}$ for $i \leq k$ and $\tx_i =0$ if $i>k$, then 
\begin{equation}
\skal{\z'-\z,\x'-\x} > d \|\x'-\x\|^2.
\label{eq:subgradgrowth}
\end{equation}
\end{lemma}

\begin{proof}
We need show that 
\begin{equation}
\skal{\pi \btz' - \btz,\pi \btx' - \btx} > d \|\pi \btx' - \btx\|^2,
\label{eq:permineq}
\end{equation}
where $\pi$ is a permutation.
For ease of notation let $ \z'= \pi \btz'$ and $\x' = \pi \btx'$.
We let the $\tI = \{i; \ \tx_i \neq 0\} = \{i; i \leq k\}$ and $I' = \{i; \ x'_i \neq 0\}$. Then
\begin{equation}
\begin{split}\skal{\z' - \btz,\x' - \btx} & =  \sum_{i \in \tI, i\in I'} (x'_i - \tx_i)^2 + \sum_{i \in \tI, i\notin I'} \tx_i(\tx_i-z'_i) \\ & + \sum_{i \notin \tI, i \in I'} x'_i(x'_i-\tz_i).\end{split}
\label{eq:skalsum}
\end{equation}
Note that 
\begin{equation} 
d\|\x'-\btx\|^2 = \sum_{i \in \tI, i\in I'} d(x'_i - \tx_i)^2 + \sum_{i \in \tI, i\notin I'}d \tx_i^2+ \sum_{i \notin \tI, i \in I'} d {x'_i}^2.
\label{eq:dnormsum}
\end{equation}
We first consider pairs of terms from the second and third sums of \eqref{eq:skalsum}. If $i \in \tI, i\notin I'$ and $j \notin \tI, j\in I'$ we have 
\begin{equation}
\tx_i(\tx_i-z'_i) + x'_j ( x'_j - \tz_j) = \tx_i^2 + {x'_j}^2 - \tx_i z'_i - x'_j \tz_j.
\end{equation}
Since $j \notin \tI$ and $i\in \tI$ we have $\tz_j < (1-2d)\tz_k \leq (1-2d)\tz_i = (1-2d)\tx_i$.
Similarly, since $j \in I'$ and $i \notin I'$ we have $z'_i \leq z'_j = x'_j$.
Therefore
\begin{eqnarray}
\tx_i z'_i \leq \tx_i x'_j \leq \frac{\tx_i^2 + {x'_i}^2}{2} & \text{ and } \\
\tx_j z_j < (1-2d) \tx_j x_i \leq (1-2d)\frac{x_i^2 + \tx_i^2}{2}
\end{eqnarray}
which gives
\begin{equation}
\tx_i(\tx_i-z'_i) + x'_j ( x'_j - \tz_j) < d(\tx_i^2 + {x'_i}^2).
\label{eq:eqtermbnd}
\end{equation}
If the number of elements in $\tI$ and $I'$ are the same the two last sums of \eqref{eq:skalsum} have the same number of terms.
Then \eqref{eq:eqtermbnd} shows that 
\eqref{eq:skalsum} larger is than \eqref{eq:dnormsum} since clearly $(x'_i-\tx_i)^2 > d (x'_i-\tx_i)^2$.
It therefore remains to consider the two cases when $I$ has more elements than $I'$ and vice versa. 

Let $k'$ be the number of elements in $I'$.
Suppose first that $\tI$ has more elements than $I'$, that is, $k>k'$.
Then the middle sums of \eqref{eq:skalsum} and \eqref{eq:dnormsum} have more terms than the third ones. Therefore we need to show that
\begin{equation}
\tx_i(\tx_i - z'_i) > d \tx_i^2,
\label{eq:firstsumest}
\end{equation}
for $i \in \tI$ and $i \notin I'$. Suppose that $\pi_{ij} = 1$, that is element $j$ of $\btz'$ is moved to element $i$ of $\z'$ by the permutation $\pi$. 
By Corollary~\ref{cor:minimizers} we have that $z'_i=\tz'_j \leq \tx'_{k'} \leq \sqrt{g_{k'}}\leq \sqrt{g_{k}}$. Since $i < k$ we have by assumption \eqref{eq:singvalsep} that $\tx_i = \tz_i > \frac{\sqrt{g_k}}{(1-d)}$.
Therefore 
\begin{equation}
\tx_i- z'_i = (1-d) \tx_i + d\tx_i-z'_i >\sqrt{g_k} + d\tx_i -z'_i \geq d\tx_i,
\end{equation} 
which gives \eqref{eq:firstsumest}.

Now suppose instead that $\tI$ has fewer elements than $I'$, that is, $k'>k$.
Then we need to show that
\begin{equation}
x'_i(x'_i -\tz_i) \geq d {x'_i}^2,
\label{eq:secondsumest}
\end{equation}
for $i \notin \tI$ and $i \in I'$. 
Suppose again that $\pi_{ij} = 1$, that is element $j$ of $\btz'$ is moved to element $i$ of $\z'$ by the permutation $\pi$. 
By Corollary~\ref{cor:minimizers} we have $x'_i = \tz'_j > \tz'_{k'} = \tx'_{k'} \geq \sqrt{g_{k'}}$. Furthermore by assumption \eqref{eq:singvalsep} we have $\tz_i < (1-d) \sqrt{g_k} \leq (1-d) \sqrt{g_{k'}}$ since $k < k'$.
Therefore
\begin{equation}
x'_i - \tz_i = (1-d)x'_i+dx'_i - \tz_i \geq (1-d)\sqrt{g_{k'}}+dx'_i - \tz_i > dx'_i.
\end{equation}
which gives \eqref{eq:secondsumest}.
\end{proof}

\begin{lemma}\label{lemma:subgradbnd2}
Suppose that $\x$ fulfills the assumptions of Lemma~\ref{lemma:subgradbnd}. If $2\z' \in \partial f^{**}(\x')$ (without any additional assumptions on the values of $\x'$ or $\z'$) then \eqref{eq:subgradgrowth} holds.
\end{lemma}
\begin{proof}
By Lemma~\ref{conj:decomp} we have $f^{**}(\x') = \sum_j \lambda_j f^{**}(\x^j)$, where $\x^j$ are points where $f^{**}(\x^j) = f(\x^j)$, that is $\x^j$ has no elements in $(0,\sqrt{g_i})$. Then by Lemma \ref{lemma:subgradbnd}, for any $\z' \in \partial f^{**}(\x') \subset \bigcap_j \partial f^{**}(\x^j)$ we have 
\begin{equation}
\skal{\z'-\z,\x^j-\x}  > d \|\x^j - \x\|^2.
\end{equation}
By convexity of $\|(\cdot) - X\|^2$ we now get
\begin{equation}
\skal{\z'-\z,\x'-\x}  > d\sum_j \lambda_j \| \x^i -\x\|^2    \geq d \| \sum_j \lambda_j \x^j -\x\|^2 = d \|\x'-\x\|^2. 
\end{equation}
\end{proof}

\begin{proof}[Proof of Theorem ~\ref{thm:statpts}]
	We will show that $\nabla h(\x') = 2(I-A^T A)\x' + 2A^T b \notin \partial f^{**}(\x')$.
	Suppose that $2\z' \in \partial f^{**}(\x')$. Since $\x$ is stationary we have $2\z + \nabla h(\x) = 0$
	\begin{equation}
	\skal{\z' + \nabla h(\x'),\x'-\x} = \skal{\z' - \z,\x' - \x} + \skal{\nabla h(\x')-\nabla h(\x),\x'-\x}.
	\label{eq:scalarprod}
	\end{equation}
	For the second term we have
	\begin{equation} \begin{split}
	\skal{\nabla h(\x')-\nabla h(\x),\x'-\x} & = \|\x'-\x\|_F^2 - \|A (\x'-\x) \|^2 \\ & \leq \delta_r \|\x-\x'\|^2, \end{split}
	\end{equation}
	if $\rank(\x-\x') \leq r$, which clearly holds if $\card(\x') \leq r-k$.
	On the other hand we also have by Lemma~\ref{lemma:subgradbnd} that
	\begin{equation}
	\skal{\z'-\z,\x'-\x} > \delta_r \|\x'-\x\|^2,
	\end{equation}
	and therefore \eqref{eq:scalarprod} is positive and $\x'$ cannot be a stationary point.

	Suppose now that $\x$ is a point that has $\card(\x) < \frac{r}{2}$. We will consider the directional derivatives along the line $\x + t{\bm v}$, where ${\bm v} = \frac{\x' - \x}{\|\x' - \x\|_F}$.
	Since $f^{**}$ is convex (and finite) the directional derivative of the objective function exists and is given by
	\begin{equation}
	\sup_{\z' \in \partial f^{**}(\x+t{\bm v})}\skal{\z'+\nabla h(\x+t{\bm v}),{\bm v}}.
	\end{equation}
	Since the $\card({\bm v}) \leq r$ it is clear by the arguments above that this is positive.
\end{proof}

\section{Proof of Theorem~\ref{thm:sufficientcond}}

\begin{proof}\label{app:main2}
We will let $\x$ be a global solution to $\min_{\card(\x) \leq k} \|A\x-\b\|$ and show that this point will be stationary under the conditions above.
To do this we need to show that $2\z \in \partial f(\x)$ for $\z = (I-A^T A)\x + A^T \b$.
We first note that since $\|A\| < 1$ the vector $\x$ will be the global minimizer of \eqref{eq:vecrelax} for the fixed-cardinality relaxation, that is, the special case $g_i = 0$ if $i\leq k$ and $g_i=\infty$ if $i>k$.
This shows that $\x$ is stationary in \eqref{eq:vecrelax} for this particular choice of $g$. In particular $\x=D_s\pi \tx$ and $\z=D_s\pi \tz$ with the same $s$ and $\pi$.
(In the matrix case the corresponding statement is that the SVD's of $X$ and $Z$ have the same $U$ and $V$ matrices.)
Furthermore, since $\card(\x) \leq k$ and $g_i = 0$ when $i\leq k$ 
it is clear from Lemma~\ref{lemma:knownz} that the $\tz_i = \tx_i$ for $i \leq k$.

To show that $\x$ is stationary for a general choice of $g$ fulfilling \eqref{eq:gconst} it is enough to show that $\sqrt{g_i} \leq \tx_i$ for $i \leq k$ and $\sqrt{g_i} \geq \tx_k$ for $i > k$ by Lemma~\ref{lemma:knownz}. 
This is however implied by the stricter constraints~\eqref{eq:singvalsep1} and we therefore proceed by proving these directly.

First we show that $\tx_i$ is close to $\ty_i$. Since $\|A\x-\b\| \leq \|A\y-\b\| = \|\eps\|$ we have
\begin{equation} 
\sqrt{1-\delta_{2k}}\|\x-\y\|  \leq \|A(\x -\y)\|  \leq \|A\x-\b\| + \|A\y-\b\| \leq 2\|\eps\|. 
\end{equation}
Therefore 
\begin{equation}
|x_i-y_i| \leq \frac{2}{\sqrt{1-\delta_{2k}}}\|\eps\|.
\label{eq:xest}
\end{equation}
Furthermore
\begin{equation} 
\|\z-\x\|  = \|A^T A (\x - \y) - A^T \eps\|  \leq \|A^T\|\|A\|\|\x - \y\| + \|A^T\|\|\eps\|  \leq \|\x - \y\| + \|\eps\|  \leq \frac{3}{\sqrt{1-\delta_{2k}}}\|\eps\|. 
\end{equation}
And since $\tx_{k+1}=0$ this means that 
\begin{equation}
\tz_{k+1} \leq \frac{3}{\sqrt{1-\delta_{2k}}}\|\eps\|.
\label{eq:zest}
\end{equation}
Now inserting the above estimates in $\z_{k+1} < (1- 2 \delta_k)\z_k$ shows (after some simplification) that this constraint holds if
\begin{equation}
\ty_k > \frac{5 - 4\delta_k}{\sqrt{1-\delta_{2k}}(1-2\delta_{2k})} \|\eps\|,
\end{equation}
which is implied by \eqref{eq:yconst} since $\delta_k \geq \delta_{2k} > 0$.
Furthermore since $\sqrt{g}_i$ is non-decreasing \eqref{eq:gconst} and \eqref{eq:xest} implies that $\tz_i = \tx_i > \frac{\sqrt{g_k}}{1-\delta_k}$ for $i \leq k$, while 
\eqref{eq:gconst} and \eqref{eq:zest} implies that $\z_i < (1-\delta_k)\sqrt{g_i}$ for $i > k $.
\end{proof}

\end{document}